\documentclass{article}
\usepackage[margin = 1in]{geometry}
\usepackage{amsmath, amssymb}
\usepackage{tikz, tikz-cd}
\usepackage[style = alphabetic, maxalphanames=4, maxnames = 4, doi = false, url = false, isbn = false, firstinits = true]{biblatex}
\DeclareFieldFormat[article, misc, incollection, inbook, inproceedings, book]{title}{#1}

\DeclareFontFamily{U}{wncy}{}
\DeclareFontShape{U}{wncy}{m}{n}{<->wncyr10}{}
\DeclareSymbolFont{mcy}{U}{wncy}{m}{n}
\DeclareMathSymbol{\Sha}{\mathord}{mcy}{"58} 

\usepackage{amsthm}
\theoremstyle{definition}
\newtheorem{defn}{Definition}[section]
\newtheorem{example}[defn]{Example}
\theoremstyle{plain}
\newtheorem{thm}[defn]{Theorem}
\newtheorem{lemma}[defn]{Lemma}
\newtheorem{cor}[defn]{Corollary}
\newtheorem{prop}[defn]{Proposition}
\theoremstyle{remark}
\newtheorem{remark}[defn]{Remark}

\newcommand{\spec}{\mathrm{Spec}\,}

\newcommand{\bl}{\mathrm{Bl}}
\DeclareMathOperator{\Frac}{\mathrm{Frac}}
\newcommand{\onto}{\twoheadrightarrow}
\newcommand{\id}{\mathrm{id}}

\newcommand{\comp}[1]{\widehat{#1}}

\newcommand{\adj}{\mathrm{adj}}
\newcommand{\codim}{\mathrm{codim}}
\newcommand{\Jac}{\mathrm{Jac}}
\newcommand{\ch}{\mathrm{char}\,}

\title{Local constancy of reduction type and related invariants for curves in $p$-adic families}
\author{Jakab Schrettner}
\begin{document}
	\maketitle
	\begin{abstract}
		We investigate the behaviour of the reduction type and related invariants of curves in families of curves over a discretely valued field. By a family, we will mean a set of curves obtained by perturbing the coefficients of the defining equations. We will show that the reduction type in these families is locally constant in the topology induced by the valuation. We also derive local constancy results for some related invariants, such as the Tamagawa number, the Birch and Swinnerton-Dyer `fudge factor' and the Galois representation.
	\end{abstract}
	\tableofcontents
	\section{Introduction}
	The purpose of this paper is to investigate curves in families over a discretely valued field $K$, mainly with respect to their reduction types and related invariants. By a family we will mean curves obtainable from each other by small perturbations in the coefficients of the equations defining the corresponding curves. 
	
	In particular we will prove the following results. In each case we let $K$ be a discretely valued field, with valuation ring $R$, which is assumed to be excellent (a mild condition that is satisfied for example if $R$ has characteristic zero, or is complete), with uniformiser $\pi\in R$ and residue field $k = R/\pi R$.
	\begin{thm}\label{intro_hyperelliptic}
		Suppose $\ch K \not=2$. Let $C/K$ be a hyperelliptic curve given by
		\[C:\,y^2 = f(x)\]
		where $f\in R[x]$ is separable. Then there is an integer $d\ge 1$ such that for any $N\ge d$, if $\widetilde{f}\in R[x]$ is separable, of the same degree as $f$, $f-\widetilde{f}\in \pi^N R[x]$, and we take the hyperelliptic curve $C'$ given by
		\[C':\,y^2 = \widetilde{f}(x),\]
		then $C$ and $C'$ have regular models with the same special fibre.
	\end{thm}
	\begin{thm}\label{intro_bihyperelliptic}
		Suppose $\ch K \not= 2$. Let $C/K$ be a bihyperelliptic curve, i.e. one given by the affine model
		\[C:\,\begin{cases}
			y^2 = f(x)\\
			z^2 = g(x)
		\end{cases}\]
		for polynomials $f,g\in R[x]$ such that $fg$ is separable. Then there is an integer $d\ge 1$ such that for any $N\ge d$, if $\widetilde{f}, \widetilde{g}$ are of the same degree as $f,g$ with $\widetilde{f}\widetilde{g}$ separable, $f-\widetilde{f}, g-\widetilde{g} \in \pi^N R[x]$, and we consider the bihyperelliptic curve 
		\[C':\,\begin{cases}
			y^2 = \widetilde{f}(x)\\
			z^2 = \widetilde{g}(x)
		\end{cases},\]
		then $C$ and $C'$ have regular models with the same special fibre.
	\end{thm}
	\begin{thm}\label{intro_projective}
		Let $C/K$ be a smooth curve embedded in $\mathbb{P}^n$ as a complete intersection, i.e. given by the vanishing of $n-1$ polynomials
		\[C: V(F_1, \ldots, F_{n-1}) \subseteq \mathbb{P}^n\]
		with $F_i\in R[X_0, \ldots, X_n]$ homogeneous. Then there is an integer $d\ge 1$ such that if $N\ge d$, and $\widetilde{F}_1, \ldots, \widetilde{F}_{n-1}$ are homogeneous polynomials of the same degree as the $F_i$, with $F_i-\widetilde{F}_i\in \pi^NR[X_0, \ldots, X_n]$, and such that the curve
		\[C': V(\widetilde{F}_1, \ldots, \widetilde{F}_{n-1}) \subseteq \mathbb{P}^n\]
		is smooth, $C$ and $C'$ have regular models with the same special fibre.	
	\end{thm}
	All of these will be special cases of Theorem $\ref{main_result}$, which says the following:
	
	\begin{thm}[$=$ Theorem \ref{main_result}]
	Let $R$ be an excellent discrete valuation ring, and $K$ its field of fractions. Let $C/K$ be a smooth projective curve with a model $\mathcal{C}/R$ and a regular closed embedding $\mathcal{C}\to \mathcal{X}$ into a regular integral $R$-scheme $\mathcal{X}$. Then there is an integer $d\ge 1$ such that if $N\ge d$, and $C'$ is a smooth projective curve with a model $\mathcal{C}'\subseteq \mathcal{X}$ which is $N$-close to $\mathcal{C}$, then $C$ and $C'$ have regular models with the same special fibre.
	\end{thm}
	The embedding $\mathcal{C}\to \mathcal{X}$ being regular means that $\mathcal{C}$ is locally given by $r = \codim(\mathcal{C}, \mathcal{X})$ equations inside $\mathcal{X}$, hence the assumption of our curve being a local complete intersection.	For more details on the notion of being $N$-close, see Section \ref{section_closeness}. Intuitively, $N$-closeness means that the curves can be given locally by integral equations which are $p$-adically close to each other. In the setups of Theorems \ref{intro_hyperelliptic}, \ref{intro_bihyperelliptic}, \ref{intro_projective}, it reduces to the conditions stated therein. We note here that even though e.g. a hyperelliptic curve is usually not a complete intersection in $\mathbb{P}^n$ (unless it is an elliptic curve), we can still find an embedding into a different space where it is locally a complete intersection (see Example \ref{2_hyperell_example}), so that we can use Theorem 5.3.
	
	In the cases where Theorem \ref{main_result} applies, we perturb the coefficients of the equations defining the curve; if the perturbations are sufficiently small, this yields a curve with the same reduction type. One may regard this as saying that the reduction type is \emph{locally constant} in these families. 
	
	As a consequence, any invariant that can be read off of the special fibre of a regular model is also locally constant: this includes invariants such as the Tamagawa number (the number of $k$-points of the component group of the Jacobian), the index (smallest positive degree of a $K$-rational divisor) and the deficiency (whether $C$ has a $K$-rational divisor of degree $g-1$). In particular we have
	\begin{cor}[= Theorem \ref{tamagawa_index_local_constancy}]
		Let $C$ be a smooth projective curve over $K$. Then for $N$ sufficiently large, and $C'$ a smooth projective curve $N$-close to $C$, the following hold:
		\begin{enumerate}
			\item if $R$ is Henselian, then $C$ is locally soluble (i.e. has a $K$-point) if and only if $C'$ is locally soluble.
			\item $C$ and $C'$ have the same Néron component group and Tamagawa number.
			\item Suppose $k$ is perfect. Then $C$ and $C'$ have the same index. If $C$ and $C'$ are geometrically connected, then one is deficient if and only if the other is.
		\end{enumerate}
	\end{cor}
	
	We also obtain similar local constancy results for some more related invariants, such as the so-called BSD fudge factor (which is a local factor appearing in the Birch and Swinnerton-Dyer conjecture, see Definition \ref{defn_fudge_factor}), and the Galois representation associated to the curves. In particular, we obtain the following:
	\begin{thm}
		Let $C$ be a smooth projective curve over $K$. Then for $N$ sufficiently large, and $C'$ a smooth projective curve $N$-close to $C$, the following hold:
		\begin{itemize}
			\item (Theorem \ref{fudge_local_constancy}) if the residue field $k = R/\pi R$ is finite, the BSD fudge factors $\left|\frac{\omega}{\omega^0}\right|$ associated to $C$ and $C'$ are equal.
			\item (Theorem \ref{galois_rep_local_constancy}, Corollary \ref{galois_rep_invariants_local_constancy}) if $K$ is a local field of characteristic zero and $C$, $C'$ have genus $g\ge 2$, then the Galois representations $H^1_{\acute{e}t}(C_{\overline{\mathbb{Q}_p}}, \mathbb{Q}_l)$ and $H^1_{\acute{e}t}(C'_{\overline{\mathbb{Q}_p}}, \mathbb{Q}_l)$ are isomorphic. In particular $C$ and $C'$ have the same local Euler factor, local root number and conductor exponent. (Here $l\not= p$ is a prime.)
		\end{itemize}
	\end{thm}
	The result on the fudge factor comes from the explicit comparison of the canonical sheaves on $N$-close models of curves, which we explain in Section \ref{section_fudge_factor}. The result on Galois representations relies on recovering the Galois representation from point counts on regular models over a finite list of extension fields, following work of Dokchitser-Dokchitser \cite{MR4898315}.
	\begin{remark}
		The BSD fudge factor is computed with reference to a basis of differentials on the curve. In the theorem above, we have to take a common basis of differentials on $C$ and $C'$, by which we mean bases that arise as the restriction of a common set of differentials on the ambient space in which both $C$ and $C'$ are embedded. (In the language of Section \ref{section_fudge_factor}, they are both good for $S$, for a set of differentials $S$ on the ambient space. See Definition \ref{defn_good_for_S}). For a more specific statement for hyperelliptic curves, see Corollary \ref{fudge_factor_hyperelliptic}.
	\end{remark}
	
	The motivation for these results is twofold. First, it enables us to perform certain `global-to-local' arguments. Consider a smooth projective curve $C$ over $\mathbb{Q}_p$. Suppose we would like to show something about a local invariant of $C$ which is known in cases of `sufficiently good reduction'. By a small perturbation of the coefficients of the equations defining $C$, we may obtain a curve $C'$ defined over the global field $\mathbb{Q}$. Thus we may consider the reduction of $C'$ at all places $q\not=p$ as well. Assuming that $C'$ has `sufficiently good reduction' away from $p$, we know our desired result holds at these places. One can then argue using a product formula or other local-global result that the result must hold at $p$. Our result guarantees that if the perturbation was sufficiently small, $C$ will have the same local invariants as $C'$, meaning that the result holds for $C$ as well. Such an argument can be found e.g. in \cite[Section 11]{parity_abelian_surfaces} to approximate curves over local fields with curves defined over global fields, in order to study Jacobians of genus 2 curves.
	
	The second motivation is computational: suppose one would like to compute the reduction type of a curve given by explicit equations in a computer algebra system. Our results imply that for this purpose, it suffices to know the coefficients of the defining equations up to some finite precision, which might make the computation easier. In fact, a result such as ours is necessary to compute reduction types over fields such as $\mathbb{Q}_p$ by a computer: as the computer can only store the coefficients defining our curve up to finite precision, local constancy ensures that this precision is enough to obtain the correct result. Our results can be used to extract how much precision is sufficient to recover the reduction type, in terms of explicit equations for the curve (see Remark \ref{remark_precision}), although we have not attempted to optimise this.
	
	For elliptic curves in Weierstrass form, local constancy of various invariants follows from Tate's algorithm \cite{tate_algorithm} (see also the exposition in \cite[Chapter IV, §9]{silverman_advanced}). It follows from the description of the algorithm that a sufficiently small perturbation in the Weierstrass equation coefficients doesn't change the special fibre of the minimal regular model, and also other invariants such as the Tamagawa number, valuation of the minimal discriminant, and conductor exponent.
	
	Liu \cite{liu2025desingularizationdoublecoversregular} describes a similar algorithm for resolving the singularities of double covers of regular surfaces, giving explicit equations for the resulting surfaces. In particular, he produces an explicit algorithm, similar to Tate's algorithm, for finding a desingularization of Weierstrass models of hyperelliptic curves. This implies a local constancy result similar to ours in these cases.
	
	For hyperelliptic curves $C: y^2 = f(x)$, Dokchitser, Dokchitser, Maistret and Morgan \cite[Theorem 1.29]{dokchitser2018arithmetic}
	show, using the machinery of cluster pictures, that various invariants of $C$ are locally constant as functions of the coefficients of $f$, including the special fibre of the minimal regular model when $C$ is semistable and the residue characteristic $p$ is odd. This was extended to curves with tame potentially semistable reduction by Faraggi and Nowell \cite{Faraggi-Nowell}.
	
	For local constancy of Galois representations, Kisin \cite{kisin_local_constancy} proves that certain families of Galois representations coming from the fibres of a proper smooth family are $p$-adically locally constant. His results are somewhat more general than ours (e.g. he doesn't assume $K$ is of characteristic 0). His proof relies on rigid analytic techniques, while our work uses only the language of schemes.
	
	Our results generalise these, in that they apply to any smooth projective curve that is presented as a local complete intersection (vanishing locus of $n$ equations in $n+1$ variables) inside some ambient space. Moreover, we have no restrictions on the residue characteristic, and no restriction on the reduction type of our curve such as semistability. However our results only give information about the special fibre of some regular model. We will not attempt to extend these results to the special fibre of the minimal regular model or the minimal normal crossings model.
	
	\subsection{Layout}
	In Section \ref{section_closeness}, we formally define what we mean by families of curves, or equivalently by a small perturbation of the coefficients defining a curve. This is the notion of $N$-closeness. We also define the related notion of formal $N$-closeness, and establish some properties of these.
	
	In Section \ref{section_automorphisms}, we show that models that are $N$-close are also formally close, by constructing appropriate formal automorphisms of the models of these curves. The results in these section follow closely the work of Cutkosy and Srinivasan \cite{Cutkosky_Srinivasan_1993} on the local equivalence of singularities, adapted to work in the setting of working over a discrete valuation ring rather than a field.
	
	In Section \ref{section_blowups}, we show that formal closeness of curves persists through blowups, by constructing formal automorphisms on blowups of models.  In Section \ref{section_results} we use this and the results of the preceding section, along with known results on resolutions of singularities, to show our result on the local constancy of reduction type in families.
	
	In Section \ref{section_fudge_factor} we give some background on the BSD fudge factor, and prove its local constancy in families, while in Section \ref{section_galois_reps} we prove the local constancy of Galois representations.
	\subsection{Notation and conventions}
	We will work in the following general setup: we let $R$ be an excellent discrete valuation ring, with field of fractions $K$, residue field $k$ and uniformiser $\pi$. Denote by $v: K^\times \to \mathbb{Z}$ the normalised valuation with valuation ring $R$. In Section \ref{section_fudge_factor}, we will also fix an absolute value $|\cdot|: K\to \mathbb{R}_{\ge 0}$ associated to the valuation $v$.
	
	The scheme $\spec R$ has a closed point $s$, and a generic point $\eta$. If $\mathcal{X}$ is an $R$-scheme, the fibre above $s$ is the special fibre $\mathcal{X}_s$, and the fibre above $\eta$ is the generic fibre $\mathcal{X}_\eta = \mathcal{X}_K$. We will always denote $R$-schemes by curly letters, and their generic fibres by the corresponding regular letter, i.e.  $R$-schemes $\mathcal{C},\mathcal{X}$ have generic fibres $\mathcal{X}_K = X, \mathcal{C}_K = C$. 
	
	Recall that a \emph{model} of a (smooth, projective) curve $C$ over $K$ is an integral, proper, flat $R$-scheme $\mathcal{C}$ along with an isomorphism $\mathcal{C}_{\eta}\cong C$. Note that some sources, for example \cite{liu2002algebraic}, require models to be normal; we however don't assume this (indeed normality of models is not preserved under general blowups, therefore we have to treat non-normal models in our arguments which rely on resolving singularities by sequences of blowups). This causes no issues: the geometric results we cite from \cite{liu2002algebraic} are phrased purely in terms of schemes and are unaffected by this change of terminology.
	
	The hat notation $\comp{(\cdot)}$ will always denote $\pi$-adic completion of an $R$-algebra.
	
	All rings and schemes will be assumed to be Noetherian, and all $R$-schemes are assumed to be of finite type.
	\subsection*{Funding}
	This work was supported by the Engineering and Physical Sciences Research Council [EP/S021590/1], the EPSRC Centre for Doctoral Training in Geometry and Number Theory (The London School of Geometry and Number Theory) at University College London.
	\subsection*{Acknowledgements}
	I would like to thank Vladimir Dokchitser for helpful discussions, guidance and comments throughout the development of this paper. I would also like to thank the referee for helpful suggestions and comments.

	\section{Families of curves}\label{section_closeness}
	We first formalise our notion of $p$-adic families, or more precisely the notion of two curves related to each other by a small perturbation of the defining equations. This is done by the notion of $N$-closeness, which is defined for given models of these curves:
	\begin{defn}\label{N-closeness}
		Let $\mathcal{Y},\mathcal{Y}'$ be $R$-schemes, both closed subschemes of an $R$-scheme $\mathcal{X}$ of codimension $r$. Also let $N\ge 1$ be an integer. We say that $\mathcal{Y}$ and $\mathcal{Y}'$ are \emph{$N$-close} if there is an affine open cover $\{\mathcal{X}_i\}_i$ of $\mathcal{X}$ such that for each $i$, 
		\[\mathcal{X}_i\cap \mathcal{Y} = V(f_1, \ldots, f_r)\qquad \text{and} \qquad \mathcal{X}_i\cap \mathcal{Y}' = V(\widetilde{f}_1, \ldots, \widetilde{f}_r)\]
		such that $r = \codim(\mathcal{Y}, \mathcal{X})$ is the codimension of $\mathcal{Y}, \mathcal{Y}'$ in $\mathcal{X}$, and $f_i-\widetilde{f}_i\in \pi^N\mathcal{O}_\mathcal{X}(\mathcal{X}_i)$.
		
		Let $C, C'$ be smooth projective curves over $K$ with respective models $\mathcal{C}, \mathcal{C}'$, which are both embedded in a regular $R$-scheme $\mathcal{X}$. If $\mathcal{C}$ and $\mathcal{C}'$ are $N$-close, then by slight abuse of language we will also say that the generic fibres $C$ and $C'$ are $N$-close.
	\end{defn}
	
	This encompasses the situation in Theorem \ref{intro_hyperelliptic} (and also \ref{intro_bihyperelliptic}, \ref{intro_projective}).
	
	Note that in the above setup, $\mathcal{Y},\mathcal{Y}'\subseteq \mathcal{X}$ are \emph{regularly embedded}, meaning that they can locally be given by the vanishing of $r = \codim(\mathcal{Y}, \mathcal{X})$ equations \cite[Definition 6.3.4]{liu2002algebraic}.
	\begin{example}\label{2_hyperell_example}
		Let $R = \mathbb{Z}_p$, and $C$ be the hyperelliptic curve of genus $g$ given by
		\[C:\;y^2 = f(x).\]
		Such a curve is not in general a complete intersection in $\mathbb{P}^n$. In fact $C$ is glued together from two affine pieces
		\begin{align*}
			C_1:\; &V(y^2 - f(x))\subseteq \mathbb{A}^2_{x,y}\qquad\qquad\text{and}\\
			C_2:\; &V(v^2 - u^{2g+2}f(1/u)) \subseteq \mathbb{A}^2_{u,v}
		\end{align*}
		along the maps $u = \frac{1}{x}, v = \frac{y}{x^{g+1}}$ . Similarly if $C'$ is the hyperelliptic curve
		\[C':\;y^2 = \widetilde{f}(x),\]
		then it is similarly glued together from two affine pieces $C'_1\subseteq \mathbb{A}^2_{x,y}$ and $C'_2\subseteq \mathbb{A}^2_{u,v}$. Note that both curves are embedded in the surface obtained by gluing the two planes $\mathbb{A}^2_{x,y}$ and $\mathbb{A}^2_{u,v}$ together via the maps above.
		
		If, without loss of generality, we take $f, \widetilde{f}\in \mathbb{Z}_p[x]$, then all the same equations define $\mathbb{Z}_p$-schemes, i.e. $C$ has a model $\mathcal{C}$ glued together from two affine pieces
		\begin{align*}
			\mathcal{C}_1 &= \spec \mathbb{Z}_p[x,y]/(y^2-f(x)) \subseteq \mathcal{X}_1 = \spec \mathbb{Z}_p[x,y]\\
			\mathcal{C}_2 &= \spec \mathbb{Z}_p[u,v]/(v^2-u^{2g+2}f(1/u)) \subseteq \mathcal{X}_2 = \spec \mathbb{Z}_p[u,v]
		\end{align*}
		again along the maps $u = \frac{1}{x}, v = \frac{y}{x^{g+1}}$. Similarly $C'$ has a model $\mathcal{C}'$ glued from similar pieces. Both models are contained in the $\mathbb{Z}_p$-scheme $\mathcal{X}$ obtained by gluing $\mathcal{X}_1$ and $\mathcal{X}_2$ along the above maps.
		
		It's easy to verify that $C$ and $C'$ are $N$-close if $f-\widetilde{f}\in p^N\mathbb{Z}_p[x]$ because e.g.
		\[\mathcal{C}\cap \mathcal{X}_1 = V(y^2-f(x))\qquad\text{and}\qquad \mathcal{C}'\cap \mathcal{X}_1 = V(y^2 - \widetilde{f}(x))\]
		and $(y^2-f(x))-(y^2-\widetilde{f}(x)) = f(x) - \widetilde{f}(x)\in p^N\mathbb{Z}_p[x, y]$. A similar relation holds in the other chart.
	\end{example}
	
	We turn to establishing some properties of $N$-closeness. 
	\begin{prop}
		Let $\mathcal{Y},\mathcal{Y}'\subseteq \mathcal{X}$ be $R$-schemes, as in Definition \ref{N-closeness}. 
		\begin{enumerate}\label{closeness_properties}
			\item The condition of $N$-closeness gets stronger with increasing $N$: if $\mathcal{Y}, \mathcal{Y}'$ are $N$-close, then they are also $M$-close for any $M\le N$.
			\item If $\mathcal{Y}$ and $\mathcal{Y}'$ are 1-close, then they have the same special fibre (as subschemes of $\mathcal{X}_s$, i.e. they are the same $k$-scheme). Alternatively $\mathcal{Y}_s$ and  $\mathcal{Y}'_s$ are cut out by the same ideal sheaf inside of $\mathcal{X}_s$.
			\item Suppose $L/K$ is an extension of discretely valued fields, with valuation rings $R = \mathcal{O}_K\subseteq \mathcal{O}_L$. If $\mathcal{Y}, \mathcal{Y}'$ are $N$-close, then 
			\[\mathcal{Y}_{\mathcal{O}_L}, \mathcal{Y}'_{\mathcal{O}_L}\subseteq \mathcal{X}_{\mathcal{O}_L}\]
			are also $N$-close. Hence $N$-closeness is preserved under extensions of discrete valuation rings.
		\end{enumerate}
	\end{prop}
	\begin{proof}
		(1) and (2) are clear from the definitions. For (3), if $\spec A = U\subseteq \mathcal{X}$ is an affine open, then $U_{\mathcal{O}_L} = \spec (A\otimes_{\mathcal{O}_K}\mathcal{O}_L)$ is an affine open on $\mathcal{X}_{\mathcal{O}_L}$. By $N$-closeness we have that
		\begin{align*}
			U\cap \mathcal{C} &= V(f_1, \ldots, f_r)\\
			U\cap \mathcal{C'} &= V(\widetilde{f}_1, \ldots, \widetilde{f}_r)
		\end{align*}
		with $f_i-\widetilde{f}_i\in \pi^N A$. Clearly it's also true that
		\begin{align*}
			U_{\mathcal{O}_L} \cap \mathcal{C}_{\mathcal{O}_L} &= V(f_1\otimes 1, \ldots, f_r\otimes 1)\\
			U_{\mathcal{O}_L} \cap \mathcal{C}'_{\mathcal{O}_L} &= V(\widetilde{f}_1\otimes 1, \ldots, \widetilde{f}_r\otimes 1)
		\end{align*}
		and we have that $f_i\otimes 1 - \widetilde{f}_i\otimes 1 \in \pi^N(A\otimes_{\mathcal{O}_K} \mathcal{O}_L)$.
	\end{proof}

	Next, we define the related notion of formal $N$-closeness. This means that locally on the ambient space $\mathcal{X}$, there are formal automorphisms taking $\mathcal{C}$ to $\mathcal{C}'$, which are congruent to the identity modulo $\pi^{N}$. By a formal automorphism, we mean an automorphism of the $\pi$-adic completions.
	
	In what follows, $\comp{(\cdot)}$ denotes $\pi$-adic completion.
	\begin{defn}\label{formal_N-closeness}
		Let $\mathcal{Y},\mathcal{Y}'$ be $R$-schemes, both closed subschemes of a regular integral $R$-scheme $\mathcal{X}$. Suppose $\mathcal{X}$ is covered by affine open subschemes $\mathcal{X}_i = \spec A_i$ such that $\mathcal{X}_i\cap \mathcal{Y} = \spec A_i/I_i$ and $\mathcal{X}_i\cap \mathcal{Y}' = \spec A_i/I'_i$. Let $N\ge 1$ be an integer. We say that $\mathcal{Y}$ and $\mathcal{Y}'$ are \emph{formally $N$-close} if there are
		\begin{enumerate}
			\item integral $R$-algebras $P_i$ with surjections $p_i:P_i\onto A_i$,
			\item automorphisms $\sigma_i: \comp{P_i}\to \comp{P_i}$ 
		\end{enumerate}
		satisfying the following properties:
		\begin{itemize}
			\item for each $i$, if $p_i^{-1}(I_i) = J_i$ and $p_i^{-1}(I_i') = J_i'$ then we have
			\[\sigma_i(\comp{J_i}) = \comp{J'_i}.\]
			\item for each $i$, 
			\[\sigma_i \equiv \id \mod \pi^{N}.\]
		\end{itemize}
	\end{defn}
	We can also establish some properties of formal $N$-closeness:
	\begin{prop}\label{formal_closeness_properties}
		Let $\mathcal{Y}, \mathcal{Y}'\subseteq \mathcal{X}$ be $R$-schemes.
		\begin{enumerate}
			\item If $\mathcal{Y}$ and $\mathcal{Y}'$ are formally $N$-close, then they are also formally $M$-close for any $M\le N$.
			\item Using the notation of Definition \ref{formal_N-closeness}, we have $I_i + \pi^NA_i = I_i'+\pi^NA_i$. In particular if $\mathcal{Y}$ and $\mathcal{Y}'$ are formally 1-close, then they have the same special fibre (again as closed subschemes of $\mathcal{X}_s$).
		\end{enumerate}
	\end{prop}
	\begin{proof}
		(1) is clear from the definition. For (2), the claim is equivalent to $J_i + \pi^NP_i = J_i' + \pi^NP_i$, and this is true because both ideals are $\pi$-adically closed, and their completions are equal.
	\end{proof}
	\begin{remark}
		Notice that both $N$-closeness and formal $N$-closeness for $\mathcal{Y}, \mathcal{Y}'\subseteq \mathcal{X}$ implies that $\mathcal{X}$ has an affine open cover $\mathcal{X}_i = \spec A_i$ such that 
		\[\mathcal{Y}\cap \mathcal{X}_i = \spec A_i/I_i\qquad\text{and}\qquad\mathcal{Y}'\cap\mathcal{X}_i = \spec A_i/I_i'\]
		for ideals $I_i, I_i'$ with $I_i+ \pi^N A_i = I_i' + \pi^N A_i$.
		
		We will later show that $N$-closeness implies formal $(N-d)$-closeness for some integer $d$ (Theorem \ref{closeness_to_formal_closeness}). It is unclear to the author whether the converse also holds, i.e. whether formal closeness also implies closeness (possibly with some loss on $N$).
	\end{remark}
	
	We have the following property of closeness and formal $N$-closeness as they relate to regularity:
	\begin{prop}\label{closeness_regularity}
		Let $\mathcal{Y}, \mathcal{Y}'\subseteq \mathcal{X}$ be  2-close or formally 2-close. Then if $x\in \mathcal{Y}_s$ is a point (hence also $x\in \mathcal{Y}'_s$), and $\mathcal{Y}$ is regular at $x$, then $\mathcal{Y}'$ is also regular at $x$.
	\end{prop}
	\begin{proof}
		Let $\mathfrak{n}, \mathfrak{n}_\mathcal{Y}, \mathfrak{n}_{\mathcal{Y}'}$ be the maximal ideals of the local rings $\mathcal{O}_{\mathcal{X}, x}, \mathcal{O}_{\mathcal{Y}, x} , \mathcal{O}_{\mathcal{Y}', x}$ respectively. Because $x$ is on the special fibre, $\pi\in \mathfrak{n}$. Both 2-closeness and formal 2-closeness implies that there are ideals $I,I'\subseteq \mathfrak{n}$ such that $I + \pi^2\mathcal{O}_{\mathcal{X}, x} = I'+\pi^2 \mathcal{O}_{\mathcal{X}, x}$ and
		\[\mathcal{O}_{\mathcal{Y}, x} = \mathcal{O}_{\mathcal{X}, x}/I \qquad\text{and}\qquad \mathcal{O}_{\mathcal{Y}', x} = \mathcal{O}_{\mathcal{X}, x}/I'.\]
		Moreover we have that
		\[\dim \mathcal{O}_{\mathcal{Y}, x} = \dim \mathcal{O}_{\mathcal{Y}', x} = \dim \mathcal{O}_{\mathcal{X}, x} - r \]
		for some integer $r\ge 1$. Hence, since $\mathcal{O}_{\mathcal{Y}, x}$ is regular, it suffices to prove that
		\[\mathfrak{n}_\mathcal{Y}/(\mathfrak{n}_\mathcal{Y})^2\cong \mathfrak{n}_{\mathcal{Y}'}/ (\mathfrak{n}_{\mathcal{Y}'})^2\]
		as $\mathcal{O}_{\mathcal{X}, x}$-modules (since then they have the same dimension).
		
		However we have
		\begin{align*}
			\frac{\mathfrak{n}_\mathcal{Y}}{(\mathfrak{n}_\mathcal{Y})^2}&\cong \frac{\mathfrak{n} + I}{\mathfrak{n}^2 + I} = \frac{\mathfrak{n}}{\mathfrak{n}^2 +I}\\
			\frac{\mathfrak{n}_{\mathcal{Y}'}}{(\mathfrak{n}_{\mathcal{Y}'})^2}&\cong \frac{\mathfrak{n} + I'}{\mathfrak{n}^2 + I'} = \frac{\mathfrak{n}}{\mathfrak{n}^2 + I'}
		\end{align*}
		and $I + \pi^2 \mathcal{O}_{\mathcal{X}, x} = I' + \pi^2 \mathcal{O}_{\mathcal{X}, x}$ implies that the denominators are equal, proving the claim.
	\end{proof}
	\section{Formal automorphisms of models}\label{section_automorphisms}
	In this section, let $\mathcal{Y}$ be an $R$-scheme with smooth generic fibre (such as a model of a smooth projective curve), embedded in some regular $R$-scheme $\mathcal{X}$. We would like to show that if an $R$-scheme $\mathcal{Y}'$ is $N$-close to $\mathcal{Y}$, then it's also formally close. A precise statement is as follows:
	\begin{thm}\label{closeness_to_formal_closeness}
	Let $\mathcal{Y}$ be an $R$-scheme with smooth generic fibre, regularly embedded in some regular $R$-scheme $\mathcal{X}$. Then there is an integer $d\ge 1$ such that if $N>d$, and $\mathcal{Y}'\subseteq \mathcal{X}$ is an $R$-scheme that is $N$-close to $\mathcal{Y}$, then $\mathcal{Y}$ and $\mathcal{Y}'$ are formally $(N-d)$-close.
	\end{thm}
	\begin{remark}
		As we will see from the proof, the integer $d$ above depends on $\mathcal{Y}, \mathcal{X}$ and also the cover realising the $N$-closeness of $\mathcal{Y}$ and $\mathcal{Y}'$.
	\end{remark}
	Since $\mathcal{X}$ is regular, $\mathcal{X}\to \spec R$ is a local complete intersection morphism \cite[Example 6.3.18]{liu2002algebraic}. Hence (by shrinking the $\mathcal{X}_i$ if necessary), we can assume that each $A_i$ is a quotient of a polynomial ring over $R$, i.e. there are surjections
	\[p_i: P_i = R[x_1, \ldots, x_{n_i}]\onto A_i,\]
	such that $K_i = \ker p_i$ is generated by a regular sequence. We will show the above theorem by constructing automorphisms of the $\comp{P_i}$. 
	
	Take an $R$-scheme $\mathcal{Y}'$ that is $N$-close to $\mathcal{Y}$. Then by assumption we have that
	\[\mathcal{Y}\cap \mathcal{X}_i = \spec A_i/I_i\qquad\text{and}\qquad\mathcal{Y}'\cap \mathcal{X}_i = \spec A_i/I_i'\]
	where $I_i = (f_1, \ldots, f_r)$ and $I_i' = (\widetilde{f}_1, \ldots, \widetilde{f}_r)$ such that $f_i-\widetilde{f}_i\in \pi^NA_i$. In other words, both $I_i$ and $I_i'$ are generated by regular sequences.
	
	Now since the kernel of $p_i: P_i \onto A_i$ is also generated by a regular sequence, so are the preimages $J_i = p_i^{-1}(I_i)$ and $J_i' = p_i^{-1}(I_i')$. Hence we are reduced to proving the following claim, a version of \cite[Theorem 5.1]{Cutkosky_Srinivasan_1993}:
	\begin{prop}\label{CS}
	Let $P = R[X_1, \ldots, X_n]$ be a polynomial ring over $R$ and let $J = (f_1, \ldots, f_m)$ be an ideal of height $m$ such that $\mathcal{X} = \spec P/J$ has smooth generic fibre. Then there is an integer $d\ge 1$ such that for any $N\ge d$, and any set of polynomials $g_1, \ldots, g_m\in P$ with $f_i-g_i\in \pi^N P$ and $J' = (g_1, \ldots, g_m)$ of height $m$, there exists an automorphism $\sigma$ of $\comp{P}$ such that $\sigma(\comp{J}) = \comp{J'}$ and $\sigma \equiv \id \mod {\pi^{N-d}}$.
	\end{prop}
	\begin{proof}
		The proof that follows is very similar to the proof of Proposition 6.1 in \cite{Cutkosky_Srinivasan_1993}, but adapted to the case of working over a discrete valuation ring $R$ rather than a field $k$. We give the full proof for completeness. We take $N\gg 0$, specifying further throughout the proof.
		
		Let $W = \left(\frac{\partial f_i}{\partial x_j}\right)$ be the $m\times n$ Jacobian matrix of partial derivatives of the $f_i$. We let $\{\tau_1, \ldots, \tau_p\}$ be exactly the $m$-element subsets of $\{1, 2, \ldots, n\}$, so that $p = \binom{n}{m}$. Also set $W(\tau_\alpha)$ to be the $m\times m$ submatrix of $W$ consisting of the rows of $W$ indexed by the elements of $\tau_\alpha$. We let $m_\alpha = \det W(\tau_\alpha)$, and $L = (m_1, \ldots, m_p)$ be the ideal generated by the $m\times m$ minors of the Jacobian.	
		
		Since $\mathcal{X}$ has smooth generic fibre, the ideal $L + (f_1, \ldots, f_m)$ contains some power $\pi^f$ of $\pi$. So we can write
		\[\pi^f = \sum_{j=1}^n b_jf_j + \sum_{\alpha=1}^p b'_\alpha m_\alpha\]
		for some $b_j, b'_\alpha\in P$. Then $g_i-f_i \in \pi^N P$ means that for each $i$,
		\begin{align*}
			g_i -f_i &\in \pi^{N-2f}\left(\sum_{j=1}^n b_jf_j + \sum_{\alpha=1}^p b_\alpha' m_\alpha\right)^2P\\
			\implies g_i-f_i&\equiv\sum_{j=1}^m b_{ij}f_j \mod{\pi^{N-2f}L^2}
		\end{align*}
		where $b_{ij}\in \pi^{N-2f}P$, and we take $N\ge 2f$. We therefore have 
		\[G \equiv (I + B)F \mod \pi^{N-2f}L^2\]
		where $G$ is the column vector $(g_1, \ldots, g_m)$, similarly for $F$, $B$ is the $m\times m$ matrix $(b_{ij})$, and $I$ is the $m\times m$ identity matrix. Now since $b_{ij}\in \pi^{N-2f}P$, the matrix $I+B$ is invertible over $\widehat{P}$, taking $N\ge 2f+1$. So we may replace $G$ by 
		\[G' = (I+B)^{-1}G,\]
		since the entries $g_i'$ of $G'$ generate the same ideal of $\widehat{P}$ as the $g_i$. Moreover the above relation gives us that $g_i'\equiv f_i \mod {\pi^{N-2f}L^2}$.
		
		Hence we may assume without loss of generality that $f_i\equiv g_i \mod {\pi^{N-2f}L^2}$. This means that we can write
		\[g_i = f_i + \sum_{\alpha, \beta=1}^p d^i_{\alpha\beta}m_\alpha m_\beta\]
		for some $d^i_{\alpha\beta}\in \pi^{N-2f}P$. If we let $D_{\alpha\beta}$ be the $m\times 1$ matrix
		\[D_{\alpha\beta} = \left[\begin{array}{c}
			d^1_{\alpha\beta}\\
			\vdots\\
			d^m_{\alpha\beta}
		\end{array}\right]\]
		then the previous equation becomes
		\[G = F + \sum_{\alpha, \beta = 1}^p D_{\alpha\beta}m_\alpha m_\beta.\tag*{($\dagger$)}\]
		
		Now let $B_{\tau_\alpha}$ be the $n\times m$ matrix whose rows corresponding to elements of $\tau_\alpha$ are the rows of $\adj(W(\tau_\alpha))$ in the same order, and the remaining rows are zero. This is done so that the resulting $n\times mp$ matrix
		\[A = \left[\begin{array}{c|c|c}
			B_{\tau_1}&\ldots&B_{\tau_p}
		\end{array}\right]\]
		satisfies 
		\[WA = \left[\begin{array}{c|c|c}
			m_1I&\ldots&m_pI\\
		\end{array}\right],\]
		where $I$ is the $m\times m$ identity matrix. Now consider the $mp\times 1$ column matrix
		\[T = \left[\begin{array}{c}
			T_1\\ \hline
			\vdots \\ \hline
			T_p
		\end{array}\right]\]
		where $T_\alpha$ is the $m\times 1$ column matrix
		\[(T_\alpha)_i= \sum_{\beta=1}^p a^i_{\alpha\beta}m_\beta\]
		for some indeterminates $a^i_{\alpha\beta}\in \widehat{P}$. Also set 
		\[AT = H = \left[\begin{array}{c}
			h_1(a^i_{\alpha\beta})\\
			\vdots\\
			h_n(a^i_{\alpha\beta})
		\end{array}\right].\]
		Using this notation, we have that 
		\[WH = WAT = \left[\begin{array}{c|c|c}
			m_1I&\ldots&m_pI\\
		\end{array}\right] T = \sum_{\alpha, \beta = 1}^p \left[\begin{array}{c}
			a^1_{\alpha\beta}\\
			\vdots\\
			a^m_{\alpha\beta}
		\end{array}\right]m_{\alpha}m_{\beta}.\tag{$\dagger\dagger$}\]
		Now we try to find the automorphism $\widehat{P}\to \widehat{P}$ taking $f_i$ to $g_i$. The automorphism will be of the form $x_i \mapsto x_i + u_i$ for some $u_i\in \widehat{P}$. So we would like to have $f_i(x + u) = g_i$. We have the Taylor expansion
		\[\left[\begin{array}{c}
			f_1(X + U)\\
			\vdots\\
			f_m(X + U)
		\end{array}\right] = F + W\left[\begin{array}{c}
			u_1\\
			\vdots\\
			u_n
		\end{array}\right] + 
		\sum_{i,j = 1}^n\left[\begin{array}{c}
			\lambda^1_{ij}\\
			\vdots\\
			\lambda^m_{ij}\\
		\end{array}\right]u_iu_j\]
		for some $\lambda_{ij}^k\in \widehat{P}$. Hence showing the existence of such a $U$ amounts to solving the system of equations
		\[G = F + W\left[\begin{array}{c}
			u_1\\
			\vdots\\
			u_n
		\end{array}\right] + \sum_{i,j = 1}^n \left[\begin{array}{c}
			\lambda^1_{ij}\\
			\vdots\\
			\lambda^m_{ij}
		\end{array}\right]u_iu_j\]
		for the $u_i$. We will set $u_i = h_i = h_i(a^k_{\alpha\beta})$, and try to solve the resulting system for the variables $a^k_{\alpha\beta}$. Thus we have 
		\[\left[\begin{array}{c}
			u_1\\
			\vdots\\
			u_n
		\end{array}\right] = H = AT\]
		and hence the equation we need to solve is
		\[G = F + WAT +  \sum_{i,j = 1}^n \left[\begin{array}{c}
			\lambda^1_{ij}\\
			\vdots\\
			\lambda^m_{ij}
		\end{array}\right]h_ih_j .\]
		Here we replace $WAT$ by the right-hand side of $(\dagger\dagger)$. Moreover we have that each $h_i$ is a $\comp{P}$-linear combination of terms $a^k_{\alpha\beta}m_\gamma$, so this equation is equivalent to
		\[G = F + \sum_{\alpha, \beta = 1}^p\left[\begin{array}{c}
			a^1_{\alpha\beta}\\
			\vdots\\
			a^m_{\alpha\beta}
		\end{array}\right]m_\alpha m_\beta +  \sum_{k,l=1}^m\sum_{\alpha, \ldots, \zeta = 1}^p M^{kl}_{\alpha\beta\gamma\delta\epsilon\zeta} a^k_{\alpha\beta}a^l_{\gamma\delta}m_\epsilon m_\zeta\]
		where the $M^{kl}_{\alpha\beta\gamma\delta\epsilon\zeta}$ are $m\times 1$ column vectors with entries in $\comp{P}$. 
		
		By $(\dagger)$, it suffices to solve the system
		\[D_{\alpha\beta} = \left[\begin{array}{c}
			a^1_{\alpha\beta}\\
			\vdots\\
			a^m_{\alpha\beta}\\
		\end{array}\right]+\sum_{\gamma, \delta, \epsilon, \zeta = 1}^p M^{kl}_{\gamma\delta\epsilon\zeta\alpha\beta}a^k_{\gamma\delta}a^l_{\epsilon\zeta}\]
		for the variables $a^k_{\alpha\beta}$. This is equivalent to the system
		\[d^i_{\alpha\beta} = a^i_{\alpha\beta} + \sum_{k,l=1}^m\sum_{\alpha, \ldots, \zeta = 1}^p (M^{kl}_{\alpha\beta\gamma\delta\epsilon\zeta})_i a^k_{\alpha\beta}a^l_{\gamma\delta}m_\epsilon m_\zeta.\]
		This system of equations has a solution by Hensel's lemma (Theorem \ref{Hensel}): indeed, $a^i_{\alpha\beta}=0$ is an approximate solution because $d^i_{\alpha\beta}\in\pi^{N-2f}P$, and the matrix of partial derivaties of the above polynomials evaluated at zero is the identity matrix. So we have a solution with $a^i_{\alpha\beta}\in \pi^{N-2f}\widehat{P}$.
		
		We thus have a map $\sigma: \widehat{P}\to \widehat{P}$ such that $\sigma(f_i) = g_i$. It is moreover given by $\sigma(x_i) = x_i + u_i$, where $u_i\equiv 0 \mod \pi^{N-2f}$, which implies that $\sigma \equiv \id \mod \pi^{N-2f}$. If $N\ge 2f + 2$, then $\sigma \equiv \id \mod \pi^2$, which implies that it is an automorphism by \cite[Lemma 5.0]{Cutkosky_Srinivasan_1993}. This proves our claim, with $d = 2f + 2$.	
	\end{proof}
	Using this, the proof of Theorem \ref{closeness_to_formal_closeness} is easy.
	\begin{proof}[Proof of Theorem \ref{closeness_to_formal_closeness}]
		The $N$-closeness gives us a cover by affine open sets $\mathcal{X}_i=\spec A_i$ such that
		\[\mathcal{Y}\cap \mathcal{X}_i = \spec A_i/I_i\qquad\text{and}\qquad\mathcal{Y}'\cap \mathcal{X}_i = \spec A_i/I_i'\]
		where $I_i = (f_1, \ldots, f_r)$ and $I_i' = (\widetilde{f}_1, \ldots, \widetilde{f}_r)$ are generated by regular sequences, and $f_i-\widetilde{f}_i\in \pi^NA_i$. 
		
		By refining this cover we can assume that for each $i$, we have a surjection $p_i: P_i\onto A_i$ with $P_i$ a polynomial ring over $R$, and the kernel of this projection is generated by a regular sequence (because $\mathcal{X}$ is regular, so $\mathcal{X}\to \spec R$ is a local complete intersection). Moreover by quasi-compactness of $\mathcal{X}$ we may take this cover to be finite. Put again $J_i = p_i^{-1}(I_i)$ and $J_i' = p_i^{-1}(I_i')$.
		
		By Proposition \ref{CS}, for each $i$ there is an integer $d_i$ such that if $N\ge d_i$, then there exists an automorphism $\sigma_i$ of $\comp{P_i}$ such that $\sigma_i(\comp{J_i}) = \comp{J_i'}$ and $\sigma_i \equiv \id \mod \pi^{N-d_i}$. Hence if we take $d = \max_i d_i$, then for $N> d$ it follows that $\mathcal{Y}$ and $\mathcal{Y}'$ are formally $(N-d)$-close.
	\end{proof}
	\begin{remark}\label{remark_closeness}
	The proof gives us a way to extract the value of $d$ that works: in the notation of the above proof, once we have our refined cover $\mathcal{X}_i = \spec A_i$ and the rings $P_i$, we can define $f_i$ to be the least integer such that $\pi^{f_i}$ lies in the Jacobian ideal of $J_i$ (i.e. the ideal generated by $J_i$ and the $m\times m$ minors of its Jacobian matrix). Then $d = \max_i 2f_i+2$ will work, by the proofs of Theorems \ref{closeness_to_formal_closeness} and \ref{CS}.
	
	Note that in many examples (such as for hyperelliptic curves, see Example \ref{2_hyperell_example}), $\mathcal{X}$ is locally isomorphic to affine space over $R$, and therefore the above refinement is easy to find.
	\end{remark}
	\section{Automorphisms of blowups}\label{section_blowups}
	One can obtain regular models of a curve by starting from an arbitrary model, and resolving its singularities through blowups (see Theorem \ref{resolution}). Our aim in this section will be to show that formal closeness of models is preserved under blowups in some sense, so that this process gives us formally close regular models of curves. First we review the relevant properties of blowups.
	\subsection{Blowups}
	Let $X$ be a scheme, and $D\subseteq X$ be a closed subscheme. The blowup of $X$ along $D$ is a scheme $\bl_D X$ along with a morphism $p: \bl_D X\to X$. We call $D$ the centre of the blowup. We will not describe the construction of the blowup in detail, but just note the properties relevant for us, following §8.1 of \cite{liu2002algebraic}.
	
	The blowup $\bl_DX$ can be computed locally, i.e if $U\subseteq X$ is an open set, and $p: \bl_D X\to X$ is the blowup, then 
	\[\bl_{D\cap U} U \cong p^{-1}(U).\]
	Thus it suffices to compute blowups of affine schemes. So suppose $X = \spec A$, and $D = \spec A/I$ for an ideal $I\subseteq A$. In this case we will also use the notation $\bl_{\spec A/I}\spec A = \bl_I A$. Pick generators $I = (h_1, \ldots, h_r)$ for $I$. Then the blowup $\bl_DX$ is the union of affine subschemes $\spec B_i$ for $i= 1, \ldots, r$ with 
	\[B_i = \frac{A[T_1, \ldots, T_r]/(h_iT_1-h_1, \ldots, h_iT_r-h_r)}{\text{$h_i^\infty$-torsion}}.\]
	(By $h_i^\infty$-torsion we mean any $x$ such that $h_i^ex=0$ for some $e\ge 0$). If $A$ is an integral domain and $h_i\not=0$, we can also express this as
	\[B_i = A\left[\frac{h_1}{h_i},\ldots, \frac{h_r}{h_i}\right]\subseteq \Frac(A).\]
	Now suppose that $Y\subseteq X$ is a closed subscheme, containing the centre of the blowup $D$. Then $\bl_D Y$ is naturally a closed subscheme of $\bl_D X$. In the affine case $X = \spec A$ integral, $Y = \spec A/J$ and $D =\spec A/I$ with $J\subseteq I$, we have that 
	\[\bl_D Y \cap \spec B_i = \spec B_i/J_i\]
	where $J_i$ is the \emph{strict transform} of $J$ in $B_i$, defined as
	\[J_i = \{x\in B_i\mid \exists e\ge 0,\; h_i^e x\in JB_i\}.\]
	Lastly, blowing up a regular scheme at a regular centre produces a regular scheme, i.e. if $X$ is regular, and $D\not=X$ is a regular closed subscheme, then $\bl_D X$ is regular.
	\subsection{Constructing automorphisms of blowups}
	In this subsection, we aim to show that formal $N$-closeness is preserved under blowups, in the following sense:
	\begin{thm}\label{blowup_formal_closeness}
		Let $\mathcal{Y}$ be an $R$-scheme embedded inside a regular integral $R$-scheme $\mathcal{X}$, and let $D\subseteq \mathcal{Y}_s$ be a regular closed irreducible subscheme. Then if $N\ge 3$ and $\mathcal{Y}'\subseteq \mathcal{X}$ is formally $N$-close to $\mathcal{Y}$, then the blowups $\bl_D \mathcal{Y}, \bl_D \mathcal{Y}'$ are formally $(N-1)$-close in $\bl_D \mathcal{X}$.
	\end{thm}
	\begin{proof}
		We essentially follow the proof of \cite[Theorem 5.1]{Cutkosky_Srinivasan_1993}, again adapted to our setting of working over a discrete valuation ring.
		
		By formal $N$-closeness we have $\mathcal{Y}_s = \mathcal{Y}'_s$, so $D\subseteq \mathcal{Y}'_s$ as well, and the blowups make sense. We also have that $\bl_D \mathcal{X}$ is regular, since $D$ and $\mathcal{X}$ are regular.
		
		By properties of blowups, we may work locally, so we can assume that $\mathcal{X} = \spec A$, with $\mathcal{Y} = \spec A/I$ and $\mathcal{Y}' = \spec A/I'$, and that there is a surjection $p: P \onto A$ with $p^{-1}(I) = J$, $p^{-1}(I') = J'$, and an automorphism $\sigma: \comp{P}\to \comp{P}$ such that $\sigma(\comp{J}) = \comp{J'}$ and $\sigma \equiv \id \mod \pi^N$.
		
		Then the closed set $D$ corresponds to a prime ideal $\mathfrak{p}\subseteq A$ with $\pi\in \mathfrak{p}$. We let $\mathfrak{p} = (h_0, h_1, \ldots, h_r)$ with $h_i\not=0$. The blowup $\bl_D \mathcal{X}$ is covered by affine patches $\spec B_i$, where
		\[B_i = A\left[\frac{h_0}{h_i}, \ldots, \frac{h_r}{h_i}\right].\]
		We have $\bl_D\mathcal{Y},\bl_D\mathcal{Y}'\subseteq \bl_D\mathcal{X}$, and moreover
		\[\bl_D\mathcal{Y}\cap \spec B_i = V(I_i)\]
		where $I_i$ is the strict transform of $I$ in $B_i$.
		Similarly $\bl_D\mathcal{Y}'\cap \spec B_i$ is given by the vanishing of the strict transform $I_i'$.
		
		It suffices to prove that for each $i$, there exists a surjection $q_i:P_i\onto B_i$ with $q_i^{-1}(I_i) = J_i, q_i^{-1}(I_i') = J_i'$, and an automorphism $\sigma_i: \comp{P_i}\to \comp{P_i}$ with $\sigma_i(\comp{J_i}) = \comp{J'_i}$. We prove this for $i=0$, for simplicity we let $h_0 = h$ and $B_0 = B$, so $\mathfrak{p} = (h, h_1, \ldots, h_r)$ and
		\[B =A\left[\frac{h_1}{h}, \ldots, \frac{h_r}{h}\right].\] 
		The strict transforms of $I$ and $I'$ inside $B$ are denoted by $I_0, I_0'$ respectively.
		
		Pick $h', h_i'\in C$ such that $p(h') = h$ and $p(h_i') = h_i$, and moreover let 
		\[\ker p = (h_{r+1}', \ldots, h_t')\]
		for some $t\ge r+1$. Then we have
		\[p^{-1}(\mathfrak{p}) = (h', h_1', \ldots, h_t')\]
		Now we set
		\[P_0 = P\left[\frac{h_1'}{h'}, \ldots, \frac{h_t'}{h'}\right]\]
		which is the coordinate ring of one of the charts on the blowup $\bl_{p^{-1}(\mathfrak{p})} P$. The strict transform of $\ker p\subseteq P$ in $P_0$ is the ideal
		\[L = \{x\in P_0\,\mid\,\exists e\ge 0, h'^ex\in \ker p\} = \left(\frac{h_{r+1}'}{h'}, \ldots, \frac{h_t'}{h'}\right)\]
		and $P_0/L\cong B$. Let $q: P_0\onto B$ be the corresponding surjection.
		
		Denote by $J_0, J_0'$ the strict transforms of $J, J'$ inside $P_0$. One easily sees that $J_0 = q^{-1}(I_0)$ and $J_0' = q^{-1}(I_0')$ as well (they both cut out the blowup of $A/I = P/J$ along the ideal $\mathfrak{p}$). It remains to construct an automorphism $\sigma_0: \comp{P_0}\to \comp{P_0}$ such that $\sigma(\comp{J_0}) = \comp{J_0'}$ and $\sigma\equiv \id \mod \pi^{N-1}$.
		
		The automorphism $\sigma: \comp{P}\to \comp{P}$ satisfies $\sigma\equiv \id \mod \pi^N$, so there is some $x\in \comp{P}$ such that
		\[\sigma(h') = h' + \pi^Nx\]
		but since $\pi\in \mathfrak{p}$, in $\comp{P_0}$ we can write $\pi = h'\tau$ for some $\tau\in \comp{P_0}$. Then we have
		\[\sigma(h') = h'(1 + y)\]
		for some $y\in \pi^{N-1}\comp{P_0}$. By a similar argument we can write
		\[\sigma(h_i') = h_i' + h'y_i\tag{$\star$}\]
		for some $y_i\in \pi^{N-1}\comp{P_0}$.
		
		Now define $\sigma_0: P_0\to \comp{P_0}$ by
		\begin{align*}
			\sigma_0|_P &= \sigma\\
			\sigma_0\left(\frac{h_i'}{h'}\right) &= \frac{h_i'/h' + y_i}{1 + y}
		\end{align*}
		One easily checks that this is well-defined:
		\begin{itemize}
			\item since $y\in \pi^{N-1}\comp{P_0}$, $1+y$ is a unit (as $N\ge 2$)
			\item we have $\sigma_0(h')\sigma_0(h_i'/h') = \sigma_0(h_i')$ by $(\star)$.
			\item $\comp{P_0}$ has no $\sigma_0(h') = h'(1+y)$-torsion, since it has no $h'$-torsion.
		\end{itemize}
		Moreover, since $\sigma_0(\pi) = \pi$, it extends to a map $\sigma_0: \comp{P_0}\to \comp{P_0}$. It is easily seen from the definition that $\sigma_0\cong \id\mod \pi^{N-1}$. Hence if $N\ge 3$, then this map is also an automorphism, by \cite[Lemma 5.0]{Cutkosky_Srinivasan_1993}.
		
		It remains to check that $\sigma_0(\comp{J_0}) = \comp{J_0'}$. Pick $x\in J_0$, then there is some $e\ge 0$ such that
		\[h'^ex \in J\]
		Then we know that $\sigma(h'^ex)\in \comp{J'}$, which is the $\pi$-adic completion of $J'$. So for all $n\ge 0$, there exists some $z\in J'$ such that
		\[\sigma(h'^ex) -z \in \pi^n\comp{P}\]
		We can also write, passing to $\comp{P_0}$,
		\[\sigma(h'^e x) = \sigma_0(h'^e x) = \sigma_0(h')^e\sigma_0(x) = h'^e(1+y)^e\sigma_0(x)\]
		So we have that
		\[h'^e(1+y)^e\sigma_0(x)-z\in \pi^N\comp{P_0}\]
		Meaning that there is an $a\in \comp{P}$ such that
		\[h'^e(1+y)^e\sigma_0(x) + \pi^N a = z\]
		As before, we can write $\pi = h'\tau$ in $\comp{P_0}$, so that if $n\ge e$ we can write
		\[h'^e((1+y)^e\sigma_0(x) + \pi^{n-e}\tau^e a) = z \in J'\]
		Which means by the definition of strict transform that
		\[(1+y)^e\sigma_0(x) + \pi^{N-e}\tau^ea \in J_0'\]
		Since such a relation holds for all $n\ge e$, we can conclude that $(1+y)^e\sigma_0(x)$ is contained in the $\pi$-adic closure of $J_0'$, i.e.
		\[(1+y)^e\sigma_0(x)\in \comp{J_0'}\]
		and since $1+y$ is a unit, we obtain $\sigma_0(x)\in \comp{J_0'}$ too.
		
		Hence we have shown that $\sigma_0(J_0) \subseteq \comp{J_0'}$. By $\pi$-adic continuity, we also obtain $\sigma_0(\comp{J_0})\subseteq \comp{J_0'}$. A similar argument with $\sigma_0^{-1}$ in place of $\sigma_0$ shows the opposite inclusion, so that $\sigma_0(\comp{J_0}) = \comp{J_0'}$ as claimed.	
	\end{proof}
	\section{Reduction types of curves in families}\label{section_results}
	
	In this section, we apply our results from previous sections to models of curves.
	\subsection{Local constancy of reduction type}
	Let $C$ be a smooth projective curve over $K$, with a model $\mathcal{C}$, and suppose $\mathcal{C}$ is regularly embedded in a regular integral $R$-scheme $\mathcal{X}$, i.e. it's locally given by $r = \codim(\mathcal{C}, \mathcal{X})$ equations. One can find a regular model of $C$ by a sequence of blowups as follows:
	\begin{thm}\label{resolution}
		Let $\mathcal{C}$ be a model of a smooth projective curve $C/K$. Assume $R$ is excellent. Then there is a sequence of morphisms
		\[\mathcal{C}_{reg} =\mathcal{C}_n \to \ldots \to \mathcal{C}_1\to \mathcal{C}_0 = \mathcal{C} \]
		such that $\mathcal{C}_{reg}$ is regular, and for each $i$, the map $\mathcal{C}_{i+1}\to \mathcal{C}_i$ is the blowup of $\mathcal{C}_i$ in some irreducible regular centre $D_i\subseteq X_i$, such that $D_i$ is contained in the special fibre of $\mathcal{C}_i$.
	\end{thm}
	\begin{proof}
		This follows from \cite[Theorem 0.1]{cossart2013canonical}.
	\end{proof}
	So take such a sequence of blowups in centres $D_i$. Also inductively put $\mathcal{X}_{i+1} = \bl_{D_i}\mathcal{X}_i$ so that $\mathcal{C}_i \subseteq \mathcal{X}_i$. 
	
	\begin{thm}\label{close_models_regular}
		There is an integer $d\ge 1$ such that if $N> d$, and $\mathcal{C}'\subseteq \mathcal{X}$ is a model a smooth projective curve $C'$ that is $N$-close to $\mathcal{C}$, then one can perform a sequence of blowups
		\[\mathcal{C}'_n \to \ldots \to \mathcal{C}'_1\to \mathcal{C}'_0 = \mathcal{C}' \]
		in centres $D_i$, and $\mathcal{C}'_n$ is regular and formally $(N-d)$-close to $\mathcal{C}_n$.
	\end{thm}
	\begin{proof}
		By Theorem \ref{closeness_to_formal_closeness}, there is a $d'\ge 1$ such that if $N> d'$, then $\mathcal{C}$ and $\mathcal{C}'$ are formally $(N-d')$-close. We let $d = d' + n + 2$, and claim this works.
		
		If $N> d$, then by Theorem \ref{closeness_to_formal_closeness}, $\mathcal{C}$ and $\mathcal{C}'$ are formally $(N-d')$-close. Thus they have the same special fibre, and $D_0\subseteq \mathcal{C}'_s$, so it makes sense to blow up $\mathcal{C}'$ in this centre. Now applying Theorem \ref{blowup_formal_closeness}, we see that $\mathcal{C}_1$ and $\mathcal{C}'_1$ are formally $(N-d'-1)$-close. Iterating this argument, we obtain that for each $i$, $\mathcal{C}_i$ and $\mathcal{C}'_i$ are formally $(N-d'-i)$-close.
		
		Hence $\mathcal{C}_n$ and $\mathcal{C}'_n$ are formally $(N-d'-n)$-close. Now since $N-d'-n\ge 2$ and $\mathcal{C}_n$ is regular, we have by Proposition \ref{closeness_regularity} that $\mathcal{C}'_n$ is regular at all points on the special fibre. Since its generic fibre is regular (by assumption it's a model of a smooth curve over $K$), we obtain that $\mathcal{C}'_n$ is regular. We also know that it's formally $(N-d)$-close to $\mathcal{C}_n$ as claimed.
	\end{proof}
	We can apply this to obtain our main result:
	\begin{thm}\label{main_result}	
	Let $R$ be an excellent discrete valuation ring, and $K$ its field of fractions. Let $C/K$ be a smooth projective curve with a model $\mathcal{C}/R$ and a regular closed embedding $\mathcal{C}\to \mathcal{X}$ into a regular integral $R$-scheme $\mathcal{X}$. Then there is an integer $d\ge 1$ such that if $N\ge d$, and $C'$ is a smooth projective curve with a model $\mathcal{C}'\subseteq \mathcal{X}$ which is $N$-close to $\mathcal{C}$, then $C$ and $C'$ have regular models with the same special fibre.
	\end{thm}
	\begin{proof}
		By Theorem \ref{close_models_regular}, $C$ and $C'$ have regular models which are formally 1-close for $N$ sufficiently large. These models have the same special fibre by Proposition \ref{formal_closeness_properties}(2).
	\end{proof}
	\begin{remark}\label{remark_precision}
	Again we can extract a value of $d$ that works: the value $d'$ from the proof of Theorem \ref{close_models_regular} above can be read off from $\mathcal{C}$ and $\mathcal{X}$ as in Remark \ref{remark_closeness}. If $n$ is the number of blowups necessary to resolve the singularities of $\mathcal{C}$, then $d = d' + n + 2$ is sufficient, as the proof shows. Note that as in Theorem \ref{closeness_to_formal_closeness}, the value of $d$ in this theorem depends not only on the curve $C$, but also on the open cover realising the $N$-closeness to $C'$.
	
	This shows also that if $L/K$ is an extension of discretely valued fields such that the extension of valuation rings $R = \mathcal{O}_K\subseteq \mathcal{O}_L$ is unramified, and $\mathcal{X}$ is smooth over $\mathcal{O}_K$, then the above theorem holds with the same value of $d$ for the base changed curve $C_L$ (meaning that if $C''$ is a smooth projective curve over $L$ which is $N$-close to $C_L$ with $N\ge d$, then $C_L$ and $C''$ have regular models with the same special fibre). Indeed this follows because the values of the $f_i$ in Remark \ref{remark_closeness} are unchanged under base changing to $\mathcal{O}_L$, and $\spec \mathcal{O}_L \to \spec \mathcal{O}_K$ is étale meaning that $\mathcal{X}_{\mathcal{O}_L}$ is regular since $\mathcal{X}\to\spec \mathcal{O}_K$ is smooth, and $\mathcal{C}_{\mathcal{O}_L}$ can be resolved in the same number of blowups as $\mathcal{C}$, namely the base changes of the blowups resolving $\mathcal{C}$ to $\mathcal{O}_L$.
	\end{remark}
	\subsection{Tamagawa number, index, deficiency}
	We now show the local constancy of various invariants that can be read off of the special fibre. The first invariant is the Tamagawa number, which is one of the local invariants appearing in the Birch and Swinnerton-Dyer conjecture.
	\begin{defn}
		Let $C$ be a smooth projective curve over $K$, and let $J$ be its Jacobian. The Néron model $\mathcal{J}$ of $J$ decomposes as
		\[0\to \mathcal{J}^0\to \mathcal{J}\to \varphi_\mathcal{J}\to 0\]
		where $\mathcal{J}^0$ is the identity component, and $\varphi_{\mathcal{J}}$ is the component group. The \emph{Tamagawa number} of $C$ is $c(C) = |\varphi_\mathcal{J}(k)|$, the number of $k$-points of the component group.
	\end{defn}
	We also look at the index and the deficiency. When our curve arises via base change from a number field, the deficiency controls the size of the 2-primary part of its Tate-Shafarevich group modulo squares, see \cite[Corollary 12]{poonen-stoll}. See also the exposition in \cite[§5.1]{parityranksjacobianscurves}.
	\begin{defn}
		Let $C$ be a smooth projective curve of genus $g$ over $K$. Its \emph{index} $I(C)$, is the smallest positive degree of a $K$-rational divisor on $C$. If $C$ is geometrically connected, we say that $C$ is \emph{deficient} if it has no $K$-rational divisor of degree $g-1$, i.e. if $I(C)$ does not divide $g-1$.
	\end{defn}
	
	We will say that a curve $C$ over $K$ is \emph{locally soluble} if it has a $K$-rational point.
	
	Our result on the local constancy of reduction type yields the following:
	\begin{thm}\label{tamagawa_index_local_constancy}
		Let $C$ be a smooth projective curve over $K$. Then for $N\ge 1$ sufficiently large, and $C'$ a smooth projective curve $N$-close to $C$, the following hold:
		\begin{enumerate}
			\item if $R$ is Henselian, $C$ is locally soluble if and only if $C'$ is locally soluble.
			\item $C$ and $C'$ have the same Néron component group, and hence the same Tamagawa number.
			\item If $k$ is perfect, then $C$ and $C'$ have the same index. If both $C$ and $C'$ are geometrically connected, then $C$ is deficient if and only if $C'$ is deficient.
		\end{enumerate}
	\end{thm}
	\begin{proof}
		In all cases, it suffices to take $N$ sufficiently large so that $C$ and $C'$ have regular models with the same special fibre, which is possible by Theorem \ref{main_result}.
		
		For local solubility, note that if $C$ has a regular model $\mathcal{C}$, then $C(K) = \mathcal{C}(R)$ by properness of $\mathcal{C}$. Because $\mathcal{C}$ is regular, each section lands in the smooth locus $\mathcal{C}^0$ of the map $\mathcal{C}\to \spec R$, see \cite[Corollary 4.4(b)]{silverman_advanced}, so $\mathcal{C}(R) = \mathcal{C}^0(R)$. Now since $R$ is Henselian, the reduction map 
		\[\mathcal{C}(R) = \mathcal{C}^0(R)\onto \mathcal{C}_s(k)\]
		is surjective, \cite[Corollary 6.2.13]{liu2002algebraic}. Hence $C(K)$ is non-empty if and only if $\mathcal{C}_s(k)$ is non-empty, i.e. local solubility can be read off of the special fibre of a regular model.
		
		For the Néron component group and Tamagawa number, these can also be read off of the special fibre of a regular model, see \cite[Theorem 1.1]{Bosch_1999}.
		
		The same is true of the index: let $\mathcal{C}$ be a regular model of $C$, and let $\mathcal{C}_s$ have components $\Gamma_1, \ldots, \Gamma_n$, with multiplicities $d_1, \ldots, d_n$. Also set $e_i = [\overline{k}\cap K(\Gamma_i): k]$, and let $I' = \mathrm{gcd}_i \{d_ie_i\}$. Then \cite[Corollary 1.3]{Bosch_1999} implies that $I' \mid I(C)$, and \cite[Théorème 3.1]{MR1385140} implies that $I(C) \mid I'$. (Note that in \cite{Bosch_1999} the authors use the quantity $r_i = [k^{s}\cap K(\Gamma_i): k]$ where $k^s$ is a separable closure, in place of our $e_i$. The $k$ perfect assumption is necessary to ensure that these are equal.)
		Hence $I(C) = I'$ depends only on $\mathcal{C}_s$. Since deficiency depends on whether the index divides $g-1$ where $g$ is the genus, we also obtain local constancy of the deficiency.
		
	\end{proof}
	
	\section{The BSD fudge factor}\label{section_fudge_factor}
	The BSD fudge factor is a local factor appearing in the Birch and Swinnerton-Dyer conjecture. Let $X$ be a smooth projective curve of genus $g$ over $\mathbb{Q}$, and $J = \Jac(X)$ is its Jacobian. The Birch and Swinnerton-Dyer conjecture asserts that the $L$-function $L(J, s)$ has an analytic continuation to $\mathbb{C}$, its order of vanishing at $s=1$ is the Mordell-Weil rank of $J$, and its leading coefficient is given by
	\[\frac{|\Sha_J|\cdot \textrm{Reg}_J}{|J(\mathbb{Q})|\cdot |J^{\vee}(\mathbb{Q})|}\cdot \Omega_J \cdot \prod_p c_p\left|\frac{\omega}{\omega^0}\right|_p \]
	where $\Omega_J$ is the real period, $c_p$ is the Tamagawa number at the prime $p$, and $\left|\frac{\omega}{\omega^0}\right|_p$ is the \emph{fudge factor} at $p$. 
	
	Both the real period and the fudge factor may be expressed in terms of differentials on the curve and its regular models. Firstly we have the isomorphism between global sections of differentials on $X$ and $J$:
	\[\Omega^1_{X/\mathbb{Q}}(X)\cong \Omega^1_{J/\mathbb{Q}}(J)\]
	(see \cite[Proposition 2.2]{Milne1986}). Choosing a basis of differentials $\omega_1, \ldots, \omega_g$ on $X$, i.e. a $\mathbb{Q}$-basis of the left-hand side above, we can express the real period as
	\[\Omega_J = \left|\int_{J(\mathbb{R})}\omega_1\wedge \ldots \wedge \omega_g\right|.\]
	Now choose a prime $p$. Let $\mathcal{X}$ be a regular model of $X$ over $\mathbb{Z}_p$, and $\mathcal{J}$ the Néron model of $J$ over $\mathbb{Z}_p$. Then we have an isomorphism between differentials on $\mathcal{J}$ and global sections of the canonical sheaf $\omega_{\mathcal{X}/\mathbb{Z}_p}$:
	\[\omega_{\mathcal{X}/\mathbb{Z}_p}(\mathcal{X}) \cong \Omega^1_{\mathcal{J}/\mathbb{Z}_p}(\mathcal{J})\]
	which are moreover compatible with the natural identifications $\omega_{\mathcal{X}/\mathbb{Z}_p}(\mathcal{X})\otimes_{\mathbb{Z}_p} \mathbb{Q}_p \cong \Omega^1_{X/\mathbb{Q}_p}(X)$ and $\Omega^1_{\mathcal{J}/\mathbb{Z}_p}(\mathcal{J})\otimes_{\mathbb{Z}_p} \mathbb{Q}_p \cong \Omega^1_{J/\mathbb{Q}_p}(J)$, see \cite[Lemma 9]{vanBommel02012022}.
	
	Now $\omega_1, \ldots, \omega_g$ is a $\mathbb{Q}_p$-basis of $\Omega^1_{X/\mathbb{Q}_p}(X)$. Moreover $\omega_{\mathcal{X}/\mathbb{Z}_p}(\mathcal{X})\subseteq \Omega^1_{X/\mathbb{Q}_p}(X)$ is a $\mathbb{Z}_p$-lattice of full rank. If we let $\omega^0_1, \ldots, \omega^0_g$ be a $\mathbb{Z}_p$-basis of $\omega_{\mathcal{X}/\mathbb{Z}_p}(\mathcal{X})$, then there is a unique $\lambda \in \mathbb{Q}_p$ such that
	\[\lambda\cdot \omega^0_1\wedge \ldots \wedge \omega^0_g = \omega_1\wedge \ldots \wedge \omega_g\]
	inside $\Omega^g_{J/\mathbb{Q}_p}(J)$. The $p$-adic valuation of $\lambda$ doesn't depend on the basis $\omega^0_i$ chosen, and the fudge factor at $p$ is
	\[\left|\frac{\omega}{\omega^0}\right|_p = |\lambda|_p.\]
	This measures how far the wedge product of our chosen basis of differentials $\omega_i$ is from being a Néron differential, i.e. one that is regular along the special fibre of the Néron model of $J$ at $p$.
	
	We record this as a definition in greater generality, working over $R$. Choose and fix an absolute value $|\cdot|:K \to \mathbb{R}_{\ge 0}$ induced by the discrete valuation $v: K^\times \to \mathbb{Z}$.
	\begin{defn}\label{defn_fudge_factor}
		Let $C$ be a smooth projective curve over $K$ of genus $g$, and let $\omega_1, \ldots, \omega_g$ be a $K$-basis of $\Omega^1_{C/K}(C)$. Let $\mathcal{C}$ be a regular model of $C$, and let $\omega^0_1, \ldots, \omega^0_g$ be an $R$-basis of the $R$-lattice $\omega_{\mathcal{C}/R}(\mathcal{C})$. There is a unique $\lambda\in K$ such that 
		\[\lambda\cdot \omega^0_1\wedge \ldots \wedge \omega^0_g = \omega_1\wedge \ldots \wedge \omega_g\]
		inside $\bigwedge^g\Omega^1_{C/K}(C)$. The \emph{BSD fudge factor} of $C$ computed with respect to $\omega_1, \ldots, \omega_g$ is the absolute value
		\[\left|\frac{\omega}{\omega^0}\right| = |\lambda|.\]
	\end{defn}
	\begin{remark}
		As before, this is well-defined because the absolute value of $\lambda$ doesn't depend on the choice of basis $\omega^0_i$ chosen. It also doesn't depend on our choice of regular model, as the image of the map $\omega_{\mathcal{C}/R}(\mathcal{C})\to \Omega^1_{C/K}(C)$ is independent of the regular model $\mathcal{C}$.
	\end{remark}
	Our goal in this section will be to prove a local constancy result for the BSD fudge factor.
	\subsection{The canonical sheaf}
	Let $\mathcal{C}$ be a regular model of the smooth projective curve $C$. We collect the relevant facts about the canonical sheaf, drawing on §6.3-6.4 of \cite{liu2002algebraic}.
	\begin{defn}
		The \emph{canonical sheaf} of $\mathcal{C}$ over $R$, denoted $\omega_{\mathcal{C}/R}$, is an invertible sheaf on $\mathcal{C}$ defined as follows: suppose $i: \mathcal{C}\hookrightarrow\mathcal{X}$ is a regular immersion into a smooth $R$-scheme $\mathcal{X}$, then
		\[\omega_{\mathcal{C}/R} = \det(\mathcal{N}_{\mathcal{C}/\mathcal{X}}) \otimes_{\mathcal{O}_\mathcal{C}} i^*(\det (\Omega^1_{\mathcal{X}/R})).\]
		Here $\det$ denotes the top exterior power of a locally free sheaf, and $\mathcal{N}_{\mathcal{C}/\mathcal{X}}$ is the normal sheaf, i.e. if $\mathcal{I}$ is the sheaf of ideals defining $\mathcal{C}\subseteq \mathcal{X}$, then $\mathcal{N}_{\mathcal{C}/\mathcal{X}} = i^*(\mathcal{I}/\mathcal{I}^2)^\vee$.
	\end{defn}
	Note that such an embedding $i$ always exists for $\mathcal{C}$ a regular model, because $\mathcal{C}\to \spec R$ is a local complete intersection morphism, and in fact we may take $\mathcal{X} = \mathbb{A}^n_R$ \cite[Example 6.3.18]{liu2002algebraic}. The definition of $\omega_{\mathcal{C}/R}$ is independent of the choice of $i$ by \cite[Lemma 6.4.5]{liu2002algebraic}.
	
	We can express the canonical sheaf locally as follows: take an affine open set $U = \spec A\subseteq \mathcal{C}$, and suppose $A = R[T_1, \ldots, T_n]/I$, where $I$ is an ideal generated by a regular sequence $F_1, \ldots, F_{n-1}$. This corresponds to a regular immersion $i: U\hookrightarrow \mathbb{A}^n_R$. The normal sheaf is given by
	\[\mathcal{N}_{U/\mathbb{A}^n} = (I/I^2)^\vee\]
	where $I/I^2$ is free of rank $n-1$ over $A$, with basis $\overline{F}_1, \ldots, \overline{F}_{n-1}$, the images of the $F_i$. Moreover $\Omega^1_{\mathbb{A}^n/R}$ is free of rank $n$, generated by $dT_1, \ldots, dT_n$. So we have that $\omega_{\mathcal{C}/R}|_U = \omega_{U/R} = \widetilde{\omega}_{A/R}$, where $\omega_{A/R}$ is the $A$-module
	\[\omega_{A/R} = \det(I/I^2)^\vee \otimes_A (\Omega^n_{\mathbb{A}^n/R}\otimes_{R[T_1, \ldots, T_n]} A)\]
	which is free of rank 1, with basis $(\overline{F}_1\wedge \ldots \wedge \overline{F}_{n-1})^\vee \otimes ((dT_1 \wedge \ldots \wedge dT_n)\otimes 1)$.
	
	We have a natural map
	\[c: \Omega^1_{\mathcal{C}/R} \to \omega_{\mathcal{C}/R}\]
	which is an isomorphism on the smooth locus of $\mathcal{C}\to \spec R$, in particular on the generic fibre \cite[Corollary 6.4.13]{liu2002algebraic}. Locally it is given as follows: if $U = \spec A\subseteq \mathcal{C}$ is an affine open set as above, then this corresponds to a map
	\begin{align*}
		\Omega^1_{A/R}&\to \omega_{A/R}\\
		dt_j &\mapsto (\overline{F}_1\wedge \ldots \wedge \overline{F}_{n-1})^\vee \otimes ((dF_1\wedge \ldots \wedge dF_{n-1}\wedge dT_j)\otimes 1)
	\end{align*}
	where $t_j$ is the image of $T_j$ in $A$.
	
	Taking sections over the generic fibre, $c$ induces an isomorphism
	\[\omega_{\mathcal{C}/R}(\mathcal{C})\otimes_R K =\omega_{\mathcal{C}/R}(C) \cong \Omega^1_{C/K}(C)\]
	which means that $\omega_{\mathcal{C}/R}(\mathcal{C})$ is an $R$-lattice of full rank in the $K$-vector space $\Omega^1_{C/K}(C)$. We will call this the lattice of \emph{regular differentials} inside $\omega^1_{K/C}(C)$. 
	
	We can characterize the set of regular differentials as the set of differentials which are regular along each component of the special fibre. We expand this in more detail: if $\Gamma\subseteq \mathcal{C}_s$ is a component of the special fibre, then the generic point $\xi$ of $\Gamma$ is a codimension 1 point of $\mathcal{C}$, hence $\mathcal{O}_{\mathcal{C}, \xi}$ is a discrete valuation ring (as $\mathcal{C}$ is regular). There is a corresponding valuation $v_\xi: K(\mathcal{C}) \to \mathbb{Z}\cup\{\infty\}$ (normalised so that it is surjective) on the function field. Now if $\tau: \omega_{\mathcal{C}/R}\mid_U \cong \mathcal{O}_U$ is any trivialization of $\omega_{\mathcal{C}/R}$ over an open set $U$ containing $\xi$, then we have an induced map
	\[ \omega_{\mathcal{C}/R, \eta} \underset{\tau_\eta}{\cong} \mathcal{O}_{\mathcal{C}, \eta} = K(\mathcal{C})\]
	with the first isomorphism coming from the trivialization.
	\begin{defn}
		We say that a differential $\alpha\in \Omega^1_{C/K, \eta}$ is \emph{regular along $\Gamma$} if applying $v_\xi$ to the image of $\alpha$ under the map
		\[\Omega^1_{C/K, \eta}\underset{c_\eta}{\xrightarrow{\sim}} \omega_{\mathcal{C}/R, \eta} \underset{\tau_\eta}{\xrightarrow{\sim}} K(\mathcal{C})\]
		yields a nonnegative number. 
	\end{defn} 
	\begin{remark}
		This is independent of the choice of trivialization $\tau$, as choosing another trivialization changes the isomorphism by a unit in $\mathcal{O}_{\mathcal{C}, \xi}^\times$.
	\end{remark}
	We then have the following characterization of the image of the restriction map $\omega_{\mathcal{C}/R}(\mathcal{C})\to \Omega^1_{C/K}(C)$:
	\begin{thm}\label{reg_diff_char}
		The image of the map $\omega_{\mathcal{C}/R}(\mathcal{C})\to \Omega^1_{C/K}(C)$ consists of the differentials $\alpha\in \Omega^1_{C/K}(C)$ that are regular along all components of the special fibre $\mathcal{C}_s$.
	\end{thm}
	\begin{remark}
		This statement is not really specific to our case, it holds in general for invertible sheaves on normal schemes.
	\end{remark}
	
	\begin{proof}
		First suppose that $\alpha \in \omega_{\mathcal{C}/R}(\mathcal{C})$. If $\xi$ is the generic point of any component of the special fibre, then $\alpha \in \omega_{\mathcal{C}/R, \xi}$, which is sent by any trivialization to an element of $\mathcal{O}_{\mathcal{C}, \xi}$. Hence $\alpha$ is regular along any component of the special fibre.
		
		Conversely suppose $\alpha\in \Omega^1_{C/K}(C)$ is regular along all components of the special fibre. Then for any generic point $\xi$ of a component, there is an element $\beta \in \omega_{\mathcal{C}/R, \xi}$ which has the same germ at $\eta$ as $\alpha$. Since $\mathcal{C}$ is integral, the restriction maps of $\omega_{\mathcal{C}/R}$ are injective, so $\alpha$ extends to a section of $\omega_{\mathcal{C}/R}$ over an open set containing $\xi$. Applying this to all components, $\alpha$ extends to a section of $\omega_{\mathcal{C}/R}$ over an open set $U$ containing $C$ and the generic points of each irreducible component of $\mathcal{C}_s$. Hence $\mathcal{C}\setminus U$ is a closed set of codimension $\ge 2$, and since $\mathcal{C}$ is regular, hence normal, $\alpha$ extends to a global section of the canonical sheaf as claimed.
	\end{proof}
	\subsection{Local constancy of the fudge factor}
	Throughout this subsection we work in the following setup: let $\mathcal{C}$ be a (not necessarily regular) model of a smooth projective curve, regularly embedded in a regular integral $R$-scheme $\mathcal{X}$, with generic fibre $\mathcal{X}_\eta = X$. Suppose $S = \{\omega_1, \ldots, \omega_g\}\subseteq \Omega^1_{K(X)/K}$ is a set of differentials.
	\begin{defn}\label{defn_good_for_S}
		Say that the model $\mathcal{C}\subseteq \mathcal{X}$ is \emph{good for $S$} if the following conditions are met:
		\begin{itemize}
			\item the generic fibre $C$ of $\mathcal{C}$ is a smooth projective curve of genus $g$,
			\item we have $S\subseteq \Omega^1_{X/K, \eta_C}$ where $\eta_C$ is the generic point of $C$,
			\item the natural restriction map $\Omega^1_{X/K, \eta_C}\onto \Omega^1_{C/K, \eta_C} =\Omega^1_{K(C)/K}$ maps $S$ to a $K$-basis of the space of global sections $\Omega^1_{K/C}(K) \subseteq \Omega^1_{K(C)/K}$.
		\end{itemize}
		In this situation we denote the images of $\omega_i$ in $\Omega^1_{K/C}(C)$ by $\omega_i|_C$.
	\end{defn}
	\begin{example}
		Let $\mathcal{X}$ be the $R$-scheme glued from $\mathcal{X}_1 = \spec R[x,y]$ and $\mathcal{X}_2 = \spec R[u,v]$ along the maps $u = \frac{1}{x}, v = \frac{y}{x^{g+1}}$ as in Example \ref{2_hyperell_example}, and let $S = \{\omega_1, \ldots, \omega_g\}\subseteq \Omega^1_{K(X)/K}$ be the set of differentials
		\[\omega_i = x^{i-1}\frac{dx}{y}.\]
		If $f\in R[x]$ is separable of degree $2g+1$ or $2g+2$, then the Weierstrass equation
		\[y^2 = f(x)\]
		defines a model $\mathcal{C}\subseteq \mathcal{X}$ of a hyperelliptic curve $C$, again as in Example \ref{2_hyperell_example}. Then $S$ is good for $\mathcal{C}$. Indeed, we have that
		\[\omega_i|_C = x^{i-1}\frac{dx}{y}\in \Omega^1_{K(C)/K}.\]
		It is well known that these form a $K$-basis of the global sections $\Omega^1_{K(C)/K}(C)$. Note that we may think of both $\omega_i$ and $\omega_i|_C$ as `the differential $x^{i-1}\frac{dx}{y}$', however the former is thought of as an element of $\Omega^1_{K(X)/K}$ and the latter as an element of $\Omega^1_{K(C)/K}$.
	\end{example}
	Now suppose $\mathcal{C}$ is good for $S$ and consider the $K$-vector space $V\subseteq \Omega^1_{K(X)/K}$ spanned by $S$. This is a $g$-dimensional space (since the $\omega_i|_C$ are linearly independent, so are the $\omega_i$). There is an obvious isomorphism
	\begin{align*}
		V &\xrightarrow{\sim} \Omega^1_{C/K}(C)\\
		\omega_i &\mapsto \omega_i|_C
	\end{align*}
	We will use this to transfer regular differentials into $\Omega^1_{K(X)/K}$.
	\begin{defn}
		Let $\mathcal{C}\subseteq \mathcal{X}$ be a model of a curve that is good for $S$. We define
		\[W_\mathcal{C} = \left\{\sum_{i=1}^g \lambda_i \omega_i \;:\; \lambda_i\in K, \sum_{i=1}^g \lambda_i \omega_i|_C \in \omega_{\mathcal{C}/R}(\mathcal{C}) \right\}\subseteq V\]
		i.e. under the isomorphism $V\cong \Omega^1_{C/K}(C)$ given by $\omega_i \mapsto \omega_i|_C$, the lattice $W_\mathcal{C}$ corresponds to $\omega_{\mathcal{C}/R}(\mathcal{C})$. 
	\end{defn}
	\begin{defn}
		Let $\mathcal{C}\subseteq \mathcal{X}$ be a model of a curve that is good for $S$, and let $\Gamma$ be a component of the special fibre $\mathcal{C}_s$. We define
		\[W_{\mathcal{C}, \Gamma} = \left\{\sum_{i=1}^g \lambda_i\omega_i\;:\; \lambda_i\in K, \sum_{i=1}^g \lambda_i\omega_i|_C \text{ is regular along $\Gamma$}\right\}\subseteq V\]
		i.e. under the isomorphism $V\cong \Omega^1_{C/K}(C)$ given by $\omega_i \mapsto \omega_i|_C$, this corresponds to the differentials in $\Omega^1_{C/K}(C)$ which are regular along $\Gamma$.
	\end{defn} 
	
	In what follows, we will derive a condition for a differential $\omega\in V$ to be regular along $\Gamma$, i.e. to have $\omega\in W_{\mathcal{C}, \Gamma}$, in case $\mathcal{C}$ is a regular model. This is a local question, so we may pass to an affine open set $U\subseteq \mathcal{X}$ containing the generic point $\xi$ of $\Gamma$ such that $U = \spec A$, and $U\cap \mathcal{C} = V(I)$. By shrinking $U$ further if necessary, we may assume that the special fibre of $U$ is irreducible with generic point $\xi$. In this local case we have $\omega_i\in \Omega^1_{\Frac(A)/K}$. Moreover by the assumption that $\omega_i$ are regular at the generic point of $\mathcal{C}$ we have in fact that $\omega_i \in \Omega^1_{A_I/K}$, hence also $V\subseteq \Omega^1_{A_I/K}$. (Here $A_I$ is the ring $A$ localized at the prime ideal $I$).
	
	We also have $\Omega^1_{K(C)/K} = \Omega^1_{\Frac(A/I)/K}$, and the natural restriction map
	\[\Omega^1_{A_{I}/K}\onto \Omega^1_{\Frac(A/I)/K}\]
	which restricts differentials to $C$. This amounts to tensoring with $A_I/IA_I$  (quotienting out by multiples of $I$), then quotienting out by elements of the form $di$, for $i\in I$.
	
	Because $\mathcal{X}$ is of finite type over $R$, we may assume that $A$ is a quotient of $B = R[T_1, \ldots, T_n]$. So $A/I = B/J$ for some ideal $J\subseteq B$. As $\mathcal{C}$ is regular, \cite[Corollary 6.3.22]{liu2002algebraic} implies that $J$ is generated by a regular sequence
	\[J = (F_1, \ldots, F_{n-1}).\]
	(The regular sequence has $n-1$ terms as $\mathcal{C}$ has relative dimension 1 over $R$).
	
	Note that a generic element of $\Omega^1_{A_I/K}$ is of the form
	\[\frac{\sum_{j=1}^n a_j \cdot dt_j}{c}\]
	where $a_j, c\in A$, $c\not\in I$ and $t_j\in A$ is the image of $T_j\in B$. Hence any element of $V$ is of this form.
	\begin{thm}
		Let $\mathcal{C}$ be regular. With notation as above, the differential
		\[\omega = \frac{\sum_{j=1}^n a_j \cdot dt_j}{c}\in V\]
		where $a_j, c\in A$ and $c\not\in I$, is regular along $\Gamma$ (i.e. $\omega \in W_{\mathcal{C}, \Gamma}$) if and only if
		\[v_I\left(\frac{\sum_{j=1}^n a_j\overline{u_{I, j}}}{c}\right)\ge 0\]
		where $v_I$ is the normalized valuation on $\Frac(A/I)$ associated to the discrete valuation ring $\mathcal{O}_{\mathcal{C}, \xi}$, and $u_{I,j}$ is the unique element of $B$ such that
		\[dF_1\wedge \ldots \wedge dF_{n-1}\wedge dT_j = u_{I,j}dT_1\wedge \ldots \wedge dT_n,\]
		and $\overline{u_{I,j}}$ is its image in $\Frac(A/I) = \Frac(B/J)$.
	\end{thm}
	\begin{proof}
		The canonical sheaf over $U\cap \mathcal{C} = \spec A/I$ is given by $\omega_{\mathcal{C}\cap U/R} = (\omega_I)^\sim$, where 
		\[\omega_I = \det(J/J^2)^\vee \otimes_{B/J} (\Omega^n_{B/R}\otimes_B B/J),\]
		which is free of rank 1 over $A/I = B/J$, and we have an explicit isomorphism
		\begin{align*}
			B/J &\cong \omega_I\\
			1 &\mapsto (\overline{F}_1\wedge \ldots \wedge \overline{F}_{n-1})^\vee
			\otimes ((dT_1\wedge \ldots\wedge dT_n)\otimes 1).
		\end{align*}
		This is a trivialization of the canonical sheaf on $\mathcal{C}\cap U$.
		
		We also have that the canonical sheaf restricted to the generic fibre is equal to the sheaf of differentials. This is encoded by the natural map
		\begin{align*}
			c:\Omega^1_{(A/I)/R} &\to \omega_I\\
			dt_j &\mapsto (\overline{F}_1\wedge \ldots \wedge \overline{F}_{n-1})^\vee \otimes ((dF_1\wedge \ldots \wedge dF_{n-1}\wedge dT_j) \otimes 1).\\
		\end{align*}
		This map becomes an isomorphism upon tensoring with $K$ (i.e. over the generic fibre), and hence also when tensoring with $\Frac(A/I)$ (i.e. on the stalk at the generic point of $\mathcal{C}$).
		
		To find the differentials in $V$ that are regular along $\Gamma$ we need to consider the composite map $\Omega^1_{K(C)/K}\xrightarrow{\sim} \omega_{\mathcal{C}/R, \eta_C} \xrightarrow{\sim} K(C)$, which in our case is the composite
		\[\Omega^1_{\Frac(A/I)/K}\cong \omega_I \otimes_{A/I} \Frac(A/I) \cong \Frac(A/I)\]
		where the first map is the above isomorphism, and the second map comes from the trivialization. One checks easily that the map is given by
		\begin{align*}
			dt_j&\mapsto(\overline{F}_1\wedge \ldots \wedge \overline{F}_{n-1})^\vee \otimes ((dF_1\wedge \ldots \wedge dF_r\wedge dT_j) \otimes 1) =\\
			&= (\overline{F}_1\wedge \ldots \wedge \overline{F}_{n-1})^\vee \otimes ((u_{I,j}dT_1\wedge\ldots\wedge dT_n)\otimes 1)\\
			&\mapsto \overline{u_{I,j}}.
		\end{align*}
		Hence the composite map sends 
		\[\omega = \frac{\sum_{j=1}^n a_j \cdot dt_j}{c}\mapsto \frac{\sum_{j=1}^n a_j\overline{u_{I, j}}}{c}\]
		and $\omega$ is regular along $\Gamma$ if the $v_I$-valuation of the latter is nonnegative, as claimed.
	\end{proof}
	
	Now take another regular model of a smooth projective curve $\mathcal{C}'\subseteq \mathcal{X}$ that is also good for $S$, and is also $N$-close or formally $N$-close to $\mathcal{C}$ for some $N\ge 1$. Then $\mathcal{C}_s = \mathcal{C}'_s$, so $\Gamma$ is also a component of the special fibre of $\mathcal{C}'$. Locally on $U = \spec A$, we have $U\cap \mathcal{C}' = \spec A/I'$. We can also write $A/I' = B/J'$, where $J'$ is generated by a regular sequence $J' = (\widetilde{F}_1, \ldots, \widetilde{F}_{n-1})$.
	
	We can repeat the above argument with $I'$ replacing $I$, so similarly a generic differential
	\[\frac{\sum_{j=1}^n a_j \cdot dt_j}{c}\in \Omega^1_{A_{I'}/K}\]
	where $a_j, c\in A$ and $c\not\in I'$, is regular along $\Gamma$ if and only if
	\[v_{I'}\left(\frac{\sum_{j=1}^n a_j\overline{u_{I', j}}}{c}\right)\ge 0\]
	where $v_{I'}$ is the valuation associated to the DVR $\mathcal{O}_{\mathcal{C'}, \xi}$, and $u_{I',j}$ is defined so that
	\[d\widetilde{F}_1\wedge \ldots \wedge d\widetilde{F}_{n-1}\wedge dT_j = u_{I',j}dT_1\wedge \ldots \wedge dT_n.\]
	
	We will need the following lemma:
	\begin{lemma}\label{differential_lemma}
		Let $\mathcal{C}, \mathcal{C}'\subseteq \mathcal{X}$ be $N$-close or formally $N$-close regular models of smooth projective curves, both good for $S$. With notation as above, the following hold:
		\begin{enumerate}
			\item[$(1)$] We have $v_I(\pi) = v_{I'}(\pi) = d$ where $d$ is the multiplicity of $\Gamma$ in $\mathcal{C}_s = \mathcal{C}'_s$.
			\item[$(2)$] We have $u_{I,j}\equiv u_{I',j} \mod \pi^{N-1}$.
			\item[$(3)$] If $U = \spec A$ is an affine open set in $\mathcal{X}$, then for any $x\in A$, if we denote $\widetilde{x}$ its image in $A/I$ and $\widehat{x}$ its image in $A/I'$, then $v_I(\widetilde{x}) = n$ and $N\ge \max\{2, n+1\}$ guarantees that $v_{I'}(\widehat{x}) = n$ too.
			\item[$(4)$] For any $\omega\in V$, there is an integer $d\ge 1$ (depending on $\omega$) such that if $N\ge d$ then
			\[\omega|_C\text{ is regular along $\Gamma$}\iff \omega|_{C'}\text{ is regular along $\Gamma$},\]
			or in other words
			\[\omega \in W_{\mathcal{C}, \Gamma} \iff \omega \in W_{\mathcal{C}', \Gamma}.\]
		\end{enumerate}
	\end{lemma}
	\begin{proof} 
		$(1)$ This is \cite[Exercise 8.3.3(a)]{liu2002algebraic}.
		
		$(2)$ Since $\Omega^n_{B/R}$ is free of rank 1 over $B$ with basis $dT_1\wedge \ldots \wedge dT_n$, it suffices to show that
		\[(u_{I,j}-u_{I',j})dT_1\wedge \ldots \wedge dT_n = dF_1 \wedge \ldots dF_{n-1}\wedge dT_j - d\widetilde{F}_1\wedge \ldots \wedge d\widetilde{F}_{n-1}\wedge dT_j \in \pi^{N-1}\Omega^n_{B/R}.\]
		However, we have $J + \pi^NB = J' + \pi^N B$ (this follows from both $N$-closeness and formal $N$-closeness), so that $\widetilde{F}_i = F_i + \pi^NG_i$. Hence
		\[d\widetilde{F}_i = dF_i + d(\pi^NG_i)\equiv dF_i \mod \pi^{N-1}\]
		by the Leibniz rule, which implies the claim.
		
		$(3)$ Let $\mathfrak{n}$ be the maximal ideal of the regular local ring $\mathcal{O}_{\mathcal{X}, \xi}$. We have $\mathcal{O}_{\mathcal{C}, \xi} = \mathcal{O}_{\mathcal{X}, \xi}/I$ and $\mathcal{O}_{\mathcal{C}', \xi} = \mathcal{O}_{\mathcal{X}, \xi}/I'$. Because $I + (\pi^N) = I' + (\pi^N)$ and $N\ge 2$, we know that $I$ and $I'$ have the same image in the cotangent space $\mathfrak{n}/\mathfrak{n}^2$. Since both $\mathcal{O}_{\mathcal{C}, \xi}$ and $\mathcal{O}_{\mathcal{C}', \xi}$ are 1-dimensional, we can choose $t\in \mathfrak{n}$ such that the class of $t$ and the image of $I$ generate $\mathfrak{n}/\mathfrak{n}^2$. Then the class of $t$ is a uniformiser in both $\mathcal{O}_{\mathcal{C}, \xi}$ and $\mathcal{O}_{\mathcal{C}', \xi}$. By shrinking the open set $U = \spec A$ we may assume that $t$ is defined over $U$.
		
		Now for an element $x\in A$, we have that $v_I(\widetilde{x}) \ge m \iff x \in I + (t^m)$ and $v_{I'}(\widehat{x}) \ge m \iff x \in I' + (t^m)$. By $(1)$ we have $\pi \in I + (t^d)$ and $\pi \in I' + (t^d)$ too. Now since $I + (\pi^N) = I' + (\pi^N)$, we must have $I + (t^m) = I' + (t^m)$ as long as $m\le Nd$.
		
		Hence if we have $v_I(\widetilde{x}) = n$ and $N\ge n+1$, then $x\in I + (t^n) = I' + (t^n)$ and $x \not \in I + (t^{n+1}) = I' + (t^{n+1})$, which implies that $v_{f'}(\widehat{x}) = n$ also.
		
		$(4)$ Express $\omega$ as
		\[\omega = \frac{\sum_{j=1}^n a_j\cdot dt_j}{c}\]
		with $a_j, c\in A$ and $c\not\in I$. By $(3)$ we have that for $N$ sufficiently large, $v_I(\widetilde{c}) = v_{I'}(\widehat{c})$ and in particular $c\not\in I'$.
		
		Now pick $N$ large enough such that $x = \sum_{j=1}^n a_j u_{I, j}$ and $c$ have the same valuation in $A/I$ and $A/I'$, and also such that $N\ge 1 + v_I(\widetilde{c})$. Also set
		\[y = \sum_{j=1}^n a_j u_{I', j}.\]
		Then we have that $x\equiv y \mod \pi^{N-1}$ by $(2)$. If $\omega|_C$ is regular along $\Gamma$, then we have that 
		\[v_I\left(\frac{\widetilde{x}}{\widetilde{c}}\right)\ge 0.\]
		We can also write
		\[\frac{\widehat{y}}{\widehat{c}} = \frac{\widehat{x}}{\widehat{c}} + \pi^{N-1}\frac{\widehat{z}}{\widehat{c}}\]
		for some $z\in A$. By assumption we have that $v_I(\widetilde{x}) = v_{I'}(\widehat{x})$ and $v_I(\widetilde{c}) = v_{I'}(\widehat{c})$. We also have $N\ge 1 + v_I(\widetilde{c}) = 1 + v_{I'}(\widehat{c})$, so $v_{I'}(\pi^{N-1}\frac{\widehat{z}}{\widehat{c}})\ge 0$ as well. This means $v_{I'}(\frac{\widehat{y}}{\widehat{c}})\ge 0$, and $\omega|_{C'}$ is regular along $\Gamma$.
		
		If $\omega$ is not regular along $\Gamma$, the same argument shows that $v_{I'}(\frac{\widehat{x}}{\widehat{c}}) = v_I(\frac{\widetilde{x}}{\widetilde{c}})<0$ and $v_{I'}(\pi^{N-1}\frac{\widehat{z}}{\widehat{c}})\ge 0$, so also $v_{I'}(\frac{\widehat{y}}{\widehat{c}})<0$ and therefore $\omega|_{C'}$ is not regular along $\Gamma$.
	\end{proof}
	
	We can now turn to showing local constancy of the lattice of regular differentials, and hence of the fudge factor.
	\begin{thm}\label{lattice_local_constancy}
		Suppose $k$ is finite. Let $\mathcal{C}\subseteq \mathcal{X}$ be a regular model of a smooth projective curve that is good for the set $S = \{\omega_1, \ldots, \omega_g\}\subseteq \Omega^1_{K(X)/K}$ of differentials. There is an integer $n\ge 1$ such that if $N\ge n$, and $\mathcal{C}'\subseteq \mathcal{X}$ is another regular model of a smooth projective curve that is good for $S$, and is $N$-close or formally $N$-close to $\mathcal{C}$, then $W_{\mathcal{C}, \Gamma} = W_{\mathcal{C}', \Gamma}$. In particular the BSD fudge factors of the generic fibres $C,C'$, computed with respect to the bases of differentials $\omega_i|_C, \omega_i|_{C'}$ are equal.
	\end{thm}
	\begin{proof}
		Theorem \ref{reg_diff_char} just states that
		\[W_{\mathcal{C}} = \bigcap_{\Gamma} W_{\mathcal{C}, \Gamma}\]
		where the intersection is over all components $\Gamma$ of $\mathcal{C}_s$.
		
		If $\mathcal{C}$ and $\mathcal{C'}$ are 1-close or formally 1-close, then they have the same special fibre, and hence in this case it suffices to show that 
		\[W_{\mathcal{C}, \Gamma} = W_{\mathcal{C'}, \Gamma}\]
		for all components $\Gamma$ of $\mathcal{C}_s = \mathcal{C}'_s$.

		Now let $B$ be an $R$-basis of the lattice $W_{\mathcal{C}, \Gamma}\subseteq V$, and let $B'$ be a set of coset representatives of the quotient $\pi^{-1}W_{\mathcal{C}, \Gamma}/W_{\mathcal{C}, \Gamma}$. Then $B'$ is a finite set because $W_{\mathcal{C}, \Gamma}\cong R^g$ implies that
		\[\pi^{-1}W_{\mathcal{C}, \Gamma}/W_{\mathcal{C}, \Gamma}\cong \pi^{-1}R^g/R^g\cong k^g\]
		and we assumed that $k$ is finite.
		
		By part $(4)$ of Lemma \ref{differential_lemma}, picking $N$ sufficiently large ensures that any element $\omega\in B\cup B'$ satisfies
		\[\omega \in W_{\mathcal{C}, \Gamma} \iff \omega \in W_{\mathcal{C}', \Gamma}.\]
		We claim that for such large $N$ (and if $N\ge 1$), we have $W_{\mathcal{C}, \Gamma} = W_{\mathcal{C'}, \Gamma}$. On the one hand both of these are $R$-lattices, and $B\subseteq W_{\mathcal{C'}, \Gamma}$ so we have $W_{\mathcal{C}, \Gamma}\subseteq W_{\mathcal{C'}, \Gamma}$. On the other hand if $W_{\mathcal{C}, \Gamma}\subsetneq W_{\mathcal{C'}, \Gamma}$, then there must be an $\omega\in W_{\mathcal{C'},\Gamma}\setminus W_{\mathcal{C}, \Gamma}$ such that $\pi\omega \in W_{\mathcal{C}, \Gamma}$. Then there is some $\omega'\in B'$ such that 
		\[\omega- \omega' \in W_{\mathcal{C}, \Gamma}\subseteq W_{\mathcal{C'}, \Gamma}\]
		by the definition of $B'$. Hence $\omega'\in W_{\mathcal{C'}, \Gamma}$ too. By our choice of $N$ we then have $\omega'\in W_{\mathcal{C}, \Gamma}$ too. So $\omega', \omega-\omega'\in W_{\mathcal{C}, \Gamma}$ but $\omega\not\in W_{\mathcal{C}, \Gamma}$, a contradiction. Hence $W_{\mathcal{C}, \Gamma} = W_{\mathcal{C}', \Gamma}$ as claimed.
		
		The fudge factor of $C$ with respect to the basis $\omega_i|_C$ can be computed from the lattice $W_{\mathcal{C}}\subseteq V$. Specifically, if $\omega_1^0, \ldots, \omega_g^0$ is an $R$-basis of $W_{\mathcal{C}}$, then the fudge factor is $|\lambda|$ where $\lambda\in K$ satisfies
		\[\lambda \cdot \omega_1^0\wedge \ldots \wedge \omega_g^0 = \omega_1\wedge \ldots \wedge \omega_g.\]
		Indeed this is just the definition of the fudge factor, transferred through the restriction isomorphism $V\cong \Omega^1_{C/K}(C)$. Similarly the fudge factor of $C'$ computed with respect to the $\omega_i|_{C'}$ is computed from $W_{\mathcal{C}'}$. The local constancy of the fudge factor then follows.
	\end{proof}	
	Combining this with our results on resolution of singularities by blowups, we can remove the assumption that our models are regular:
	\begin{thm}\label{fudge_local_constancy}
		Suppose $k$ is finite. Let $\mathcal{C}\subseteq \mathcal{X}$ be a model of a smooth projective curve that is good for the set $S = \{\omega_1, \ldots, \omega_g\}\subseteq \Omega^1_{K(X)/K}$ of differentials. There is an integer $n\ge 1$ such that if $N\ge n$, and $\mathcal{C}'\subseteq \mathcal{X}$ is another model of a smooth projective curve that is good for $S$, and is $N$-close to $\mathcal{C}$, then the BSD fudge factors of $C$ and $C'$, computed with respect to the bases of differentials $\omega_i|_C, \omega_i|_{C'}$ respectively, are equal.
	\end{thm}
	\begin{proof}
		By Theorem \ref{resolution} there is a sequence of blowups 
		\[\mathcal{C}_n \to \ldots \to \mathcal{C}_1\to \mathcal{C}_0 = \mathcal{C}\]
		in regular centres $D_i$ on the special fibre, such that $\mathcal{C}_n$ is regular. By Theorem \ref{close_models_regular}, there is an $a\ge 1$ such that if $N> a$ and $\mathcal{C}'\subseteq \mathcal{X}$ is $N$-close to $\mathcal{C}$, then we can perform a sequence of blowups
		\[\mathcal{C}'_n\to \ldots \to \mathcal{C}'_1 \to \mathcal{C}'_0 = \mathcal{C}'\]
		in centres $D_i$, and $\mathcal{C}'_n$ is regular and formally $(N-a)$-close to $\mathcal{C}_n$.
		
		Putting $\mathcal{X}_0 = \mathcal{X}$ and $\mathcal{X}_{i+1} = \bl_{D_i}\mathcal{X}_i$ so that $\mathcal{C}_i,\mathcal{C}'_i\subseteq \mathcal{X}_i$. Since all the blowups are centred on the special fibre, the generic fibres are unchanged, and in particular we have the identification on differentials
		\[\Omega^1_{K(X_i)/K}\cong \Omega^1_{K(X)/K}.\]
		So we can regard $S$ as a set of differentials on the $\mathcal{X}_i$ for any $i$. Since $\mathcal{C}, \mathcal{C}'$ are good for $S$, so are $\mathcal{C}_i$ and $\mathcal{C}'_i$ for any $i$, since the condition depends only on the generic fibre. Hence $\mathcal{C}_n, \mathcal{C}'_n$ are regular models of smooth projective curves, which are $(N-a)$-close, and both good for the set of differentials $S$. Hence for $N$ sufficiently large, the BSD fudge factors computed with respect to $S$ are equal by Theorem \ref{lattice_local_constancy}.
	\end{proof}
	
	We make this more specific in the case of hyperelliptic curves:
	\begin{cor}\label{fudge_factor_hyperelliptic}
		Let $C/K$ be a hyperelliptic curve of genus $g$ given by the Weierstrass equation
		\[C: y^2 = f(x)\]
		with $f\in R[x]$ separable. There exists an integer $n\ge 1$ such that if $N\ge 1$, and $C'$ is the hyperelliptic curve
		\[C': y^2 = \widetilde{f}(x)\]
		where $\widetilde{f}\in R[x]$ is separable of the same degree as $f$, and $f\equiv \widetilde{f} \mod \pi^N$, then $C$ and $C'$ have the same BSD fudge factors, both computed with respect to the basis of differentials $S = \{x^i\frac{dx}{y}\mid 0\le i \le g-1\}$.
	\end{cor}
	\begin{proof}
		We know that the curves $C, C'$ are $N$-close, and that both are good for $S$. Hence this is a special case of Theorem \ref{fudge_local_constancy}.
	\end{proof}
	\section{Galois representations}\label{section_galois_reps}
	In this section we prove a local constancy result for Galois representations in families of curves. The precise result is the following:
	\begin{thm}\label{galois_rep_local_constancy}
		Let $K$ be a local field of characteristic 0. Let $\mathcal{C}$ be a regular model of the smooth projective curve $C$ of genus $g\ge 2$, and suppose there is a regular immersion $\mathcal{C}\subseteq \mathcal{X}$ with $\mathcal{X}$ a smooth $\mathcal{O}_K$-scheme. Then there is an integer $d\ge 1$ such that if $N\ge d$, and $C'$ is a smooth projective curve with a model $\mathcal{C}'\subseteq \mathcal{X}$ which is $N$-close to $\mathcal{C}$, then the Galois representations
		\[H^1_{\acute{e}t}(C_{\overline{K}}, \mathbb{Q}_l)\qquad\text{and}\qquad H^1_{\acute{e}t}(C'_{\overline{K}}, \mathbb{Q}_l)\]
		are isomorphic (here $l\not=p$ is a prime).
	\end{thm}
	\begin{proof}
		The proof uses a result of Tim and Vladimir Dokchitser \cite[Theorem 2(iii)]{MR4898315}, saying that the $G_K$-representation $H^1_{\acute{e}t}(C_{\overline{K}}, \mathbb{Q}_l)$ is uniquely determined by point counts $\# \overline{\mathcal{C}}_L(\mathbb{F}_L)$, for a finite list of extensions $L/K$. Here $\mathcal{C}_L$ is a regular model of $C_L$ over $L$, with special fibre $\overline{\mathcal{C}}_L$, and $\mathbb{F}_L$ is the residue field of $L$.
		
		Hence it suffices to show that for any fixed finite extension $L/K$, $C_L$ and $C'_L$ have regular models with the same number of points on their special fibre. We have that $C_L$ and $C'_L$ are $N$-close, since $N$-closeness is preserved under extensions of discrete valuation rings by Proposition \ref{closeness_properties}(3).	Note also that since $\mathcal{X}$ is smooth over $\mathcal{O}_K$, $\mathcal{X}_{\mathcal{O}_L}$ is regular. Hence by Theorem \ref{main_result}, for $N$ sufficiently large, we have that $C_L$ and $C'_L$ have regular models with the same special fibre. In particular these regular models have the same number of $\mathbb{F}_L$-points on their special fibre, proving the claim.
	\end{proof}
	\begin{remark}
		In \cite{MR4898315}, Theorem 2(iii) is phrased in terms of the minimal regular model $\mathcal{C}_L$ of $C$ over $L$. However, the proof goes through even if we replace $\mathcal{C}_L$ by an arbitrary regular model. Indeed, the proof relies on the isomorphism
		\[H^1_{\acute{e}t}(C_{\overline{L}}, \mathbb{Q}_l)^{I_L}\cong H^1_{\acute{e}t}((\overline{\mathcal{C}}_L)_{\overline{\mathbb{F}}_L}, \mathbb{Q}_l),\]
		and the Grothendieck-Lefschetz trace formula to express the Frobenius traces on étale cohomology as point counts. Both of these hold in the case of $\mathcal{C}_L$ a proper regular model.
	\end{remark}
	As an easy corollary, we also obtain the following:
	\begin{cor}\label{galois_rep_invariants_local_constancy} 
		Let $K$ be a local field of characteristic 0. Let $\mathcal{C}$ be a regular model of the smooth projective curve $C$ of genus $g\ge 2$, and suppose there is a regular immersion $\mathcal{C}\subseteq \mathcal{X}$ with $\mathcal{X}$ a smooth $\mathcal{O}_K$-scheme. Then there is an integer $d\ge 1$ such that if $N\ge d$, and $C'$ is a smooth projective curve with a model $\mathcal{C}'\subseteq \mathcal{X}$ which is $N$-close to $\mathcal{C}$, then the curves $C$ and $C'$ have the same local Euler factor, local root number and conductor exponent.
	\end{cor}
	\begin{proof}
		All the listed invariants are functions of the Galois representation, hence the claim follows by Theorem \ref{galois_rep_local_constancy}.
	\end{proof}
	\appendix
	\section{Appendix: Hensel's lemma}
	In one of our proofs, we shall need the following version of Hensel's lemma, for $\pi$-adically complete non-local rings.
	\begin{thm}[Hensel's lemma]\label{Hensel}
		Let $A$ be a ring, complete with respect to an ideal $I$. Suppose that $f_1, \ldots, f_n\in A[x_1, \ldots, x_n]$ are $n$ polynomials in $n$ variables, and suppose we have $a_1, \ldots, a_n\in A$ such that 
		\[f_i(a_1, \ldots, a_n)\in I^d \text{ for each }i\]
		for some integer $d\ge 1$, and that the $n\times n$ matrix
		\[J = \left(\frac{\partial f_i}{\partial x_j}(a_1, \ldots, a_n)\right)_{i,j = 1}^n\]
		is invertible over $A$. Then there exist $b_1, \ldots, b_n\in A$ such that $b_i\equiv a_i \mod I^d$ and 
		\[f_i(b_1, \ldots, b_n) = 0\text{ for each }i.\]
	\end{thm}
	\begin{proof}
		Define the function $F$ on $A^n$ by
		\[F\left(\left[\begin{array}{c}
			t_1\\
			\vdots\\
			t_n
		\end{array}\right]\right) = \left[\begin{array}{c}
			f_1(t_1, \ldots, t_n)\\
			\vdots\\
			f_n(t_1, \ldots, t_n)
		\end{array}\right].\]
		We change variables as follows: let $g_1, \ldots, g_n\in A[x_1, \ldots, x_n]$ be polynomials such that for $X = (x_1, \ldots, x_n)$ and $\mathbf{a} = (a_1, \ldots, a_n)$ we have
		\[F(J^{-1}(X + \mathbf{a})) = \left[\begin{array}{c}
			g_1(X)\\
			\vdots\\
			g_n(X)
		\end{array}\right] = G(X).\]
		Then by assumption $g_i(0, \ldots, 0) \in I^d$ for all $i$, and the matrix $(\frac{\partial g_i}{\partial x_j}(0, \ldots, 0))$ is the $n\times n$ identity matrix. If we can find $\mathbf{b}' = (b_1', \ldots, b_n')$ such that $g_i(\mathbf{b}') = 0$ for all $i$, and $\mathbf{b}'\in I^dA^n$, then $\mathbf{b} = J\mathbf{b}' + \mathbf{a}$ satisfies the conditions of the theorem. 
		
		Hence we may assume without loss of generality that $\mathbf{a} = (0, \ldots, 0)$ and that $J$ is the identity matrix.
		
		We define the sequence $\mathbf{a}_m$ of elements of $A^n$ by the recurrence
		\[\mathbf{a}_0 = \mathbf{a}\qquad\text{and}\qquad \mathbf{a}_{m+1} = \mathbf{a}_m - F(\mathbf{a}_m).\]
		We claim that $\mathbf{a}_i$ converges to $\mathbf{b} = (b_1, \ldots, b_n)$ which satisfies the requirements of the theorem.
		
		It is clear that for all $m$, $\mathbf{a}_m \in I^dA^n$. We can show by induction that $\mathbf{a}_m \equiv \mathbf{a}_{m+1}\mod I^{d+m}$ for all $m\ge 0$. For $m=0$, this is true by assumption (as $F(\mathbf{a})\in I^dA^n$). Now suppose $m\ge 1$ and the claim holds for all integers up to $m$. We have
		\begin{align*}
			\mathbf{a}_{m+1}-\mathbf{a}_m &= (\mathbf{a}_m - F(\mathbf{a}_m)) - (\mathbf{a}_{m-1} - F(\mathbf{a}_{m-1})) =\\
			&= (\mathbf{a}_m - \mathbf{a}_{m-1}) - (F(\mathbf{a}_m) - F(\mathbf{a}_{m-1})).
		\end{align*}	
		We will use the Taylor expansion
		\[F(X) = F(0) + X + \sum_{i,j=1}^n x_ix_jR_{ij},\]
		which holds for some $R_{ij}\in A^n$ (remember $J$ is the identity). Plugging in $\mathbf{a}_m$ and $\mathbf{a}_{m-1}$, and using $\mathbf{a}_{m}\equiv \mathbf{a}_{m-1} \mod I^{d+m-1}$ we obtain that
		\[F(\mathbf{a}_m)-F(\mathbf{a}_{m-1}) \equiv \mathbf{a}_m - \mathbf{a}_{m-1} \mod I^{2(d+m-1)}\]
		and since $2(d+m-1)\ge d + m$, we obtain that $\mathbf{a}_{m+1}-\mathbf{a}_m \in I^{d+m}$.
		
		Hence since $A$ is $I$-adically complete, we obtain a limit $\mathbf{b} = \lim_{m\to \infty} \mathbf{a}_m$. Since $\mathbf{a}_m \in I^dA^n$ for all $m$, we have $\mathbf{b}\in I^dA^n$, and taking the limit of $F(\mathbf{a}_m) = \mathbf{a}_{m}-\mathbf{a}_{m+1}$ as $m\to \infty$ we obtain that $F(\mathbf{b}) = 0$, as desired.
	\end{proof}
	\printbibliography
\end{document}